\documentclass{article}

\usepackage[utf8]{inputenc}
\usepackage{amsmath}
\usepackage{amsfonts}
\usepackage{amsthm}
\usepackage{amssymb}
\usepackage{array}
\usepackage[hidelinks]{hyperref}
\usepackage[english]{babel}
\newtheorem{definition}{Definition}
\newtheorem{lemma}{Lemma}
\newtheorem{proposition}{Proposition}
\newtheorem{remark}{Remark}
\newtheorem{theorem}{Theorem}

\newtheorem{corollary}{Corollary}

\usepackage{booktabs}
\usepackage{bbm}
\usepackage{float}
\usepackage{longtable}
\newcommand{\enstq}[2]{\left\{#1\mathrel{}\middle|\mathrel{}#2\right\}}
\newcommand{\norm}[1]{\left\|#1\right\|}
\newcommand{\N}{\mathbb{N}}
\newcommand{\Z}{\mathbb{Z}}
\newcommand{\R}{\mathbb{R}}
\newcommand{\duality}[2]{\left\langle #1,#2\right\rangle}
\newcommand{\inner}[2]{\left( #1,#2\right)}
\newcommand{\abs}[1]{\left\lvert #1 \right\rvert}

\newcommand{\Cinf}{C^{\infty}}
\newcommand{\isdef}{\mathrel{\mathop:}=}
\newcommand{\opFromTo}[3]{#1 : #2 \longrightarrow #3}
\usepackage{todonotes}
\usepackage{subcaption}
\usepackage{scalerel,stackengine}
\stackMath
\newcommand\reallywidecheck[1]{%
	\savestack{\tmpbox}{\stretchto{%
			\scaleto{%
				\scalerel*[\widthof{\ensuremath{#1}}]{\kern-.6pt\bigwedge\kern-.6pt}%
				{\rule[-\textheight/2]{1ex}{\textheight}}
			}{\textheight}%
		}{0.5ex}}%
	\stackon[1pt]{#1}{\scalebox{-1}{\tmpbox}}%
}
\newcommand\conj[1]{\overline{#1}}
\title{Pseudo-differential analysis of the Helmholtz layer potentials on open curves} 
\date{}
\author{Martin Averseng \thanks{Centre de Mathématiques Appliquées (UMR 7641), Ecole Polytechnique, Route de Saclay, 91128  PALAISEAU Cedex.}}

\begin{document}

\maketitle

\begin{abstract}
	We introduce two new classes of pseudo-differential operators on open curves. They correspond via a change of variables to subclasses of the periodic  pseudo-differential operators, which respectively stabilize even and odd functions. The resulting symbolic calculus can be applied to the analysis of the Helmholtz weighted layer potentials on open curves. In particular, we build some low order parametrices of the layer potentials which take the form of square roots of tangential operators. This gives some foundation for the construction of efficient preconditioners for the Helmholtz scattering problem by a screen in 2D. 
\end{abstract}

\section*{Introduction}

The Helmholtz scattering Dirichlet and Neumann problems by a 1-dimensional screen in a 2D context can be recast into first-kind integral equations involving two boundary integral operators on an open curve $\Gamma$, namely the single-layer potential $S_k$ and the hypersingular operator $N_k$. More precisely, some ``weighted" versions of those layer potentials, first introduced in \cite{bruno2012second}, are considered here: 
\[S_{k,\omega}\phi := S_k \left(\frac{\phi}{\omega}\right)\,, \quad N_{k,\omega} \phi := N_k(\omega\phi)\]
where $\omega$ is a ``weight" defined on the curve $\Gamma$.  
In this work, we introduce a theoretical framework to analyze those operators in terms of pseudo-differential properties. To this end two classes of pseudo-differential operators are defined, containing $S_{k,\omega}$ and $N_{k,\omega}$ and allowing to see them as operators of order $-1$ and $1$ respectively. The symbolic calculus available in those classes allows to build some simple low-order parametrices for $S_{k,\omega}$ and $N_{k,\omega}$. This provides the theoretical foundation for the preconditioning strategy exposed in \cite{alougesAverseng}. Although the behavior of those preconditioners is not fully explained, we believe that the work proposed here gives convincing arguments for their practical efficiency.

The pseudo-differential analysis presented here differs significantly from classical works dealing with pseudo-differential operators on singular manifolds \cite{melrose,schulze1,schulze2}, which rely on Mellin transforms. One notable exception is \cite[Chap. 11]{saranen2013periodic}, where the analysis of the layer potentials on open curves is brought back to the analysis of periodic pseudo-differential operators (see e.g. \cite{turunen1998symbol}) using a change of variables. Parametrices of those operators are derived and discretized using truncation of Fourier series. This method is very well suited for discretization with trigonometric polynomials. 

In contrast, here we describe a way to bring the analysis ``back to the original curve" and the parametrices are defined intrinsically, involving only tangential differential operators defined on the curve. The resulting operators can thus be easily discretized by any standard numerical method (such as piecewise linear polynomials in \cite{alougesAverseng}). The main difficulty throughout is that the inverse change of variables is singular. This brings some unusual behavior, as the non-uniqueness of the symbol for instance. 

The outline is as follows. In the first section, we study two interpolating scales of Hilbert spaces $T^s$ and $U^s$ introduced in \cite{alougesAverseng} that play an analogous role as Sobolev spaces in standard pseudo-differential theory. The new classes of pseudo-differential operators on open curves, respectively in the scales $T^s$ and $U^s$ are introduced in the second section. In the third section, we introduce the operators $S_{k,\omega}$ and $N_{k,\omega}$ and give some properties needed for the analysis. The low-order parametrices for $k \neq 0$ are studied in the last section.  

Throughout all this article, the letter $C$ denotes a generic constant in estimates of the form $ a \leq C b$. Its value may change from line to line, but is independent of the relevant parameters defining $a$ and $b$.

\section{Spaces \texorpdfstring{$T^s$ and $U^s$}{Ts and Us}}

\subsection{Sobolev spaces of periodic even and odd functions}

We consider the torus $\mathbb{T}_{2\pi} = \mathbb{R}/2\pi \mathbb{Z}$, and denote by $L^2_{per}$ the set of square integral functions on $\mathbb{T}_{2\pi}$. This is a Hilbert space for the scalar product
\[\inner{u}{v}_{L^2_{per}} \isdef \frac{1}{2\pi}\int_{-\pi}^{\pi} u(\theta) \conj{v(\theta)} d\theta\,.\]
Any function in $u \in L^2_{per}$ can be expanded in Fourier series
\[u = \sum_{n \in \mathbb{Z}} \mathcal{F}u(n) e_n \quad  \textup{in } L^2_{per}\]
where $e_n : \theta \mapsto e^{in\theta}$. The coefficients $\mathcal{F}u(n)$ are obtained by orthogonal projection $\mathcal{F}u(n) =(u,e_n)_{L^2_{per}}.$ For all $s \geq 0$, the Sobolev space $H^s_{per}$ is the set of functions $u\in L^2_{per}$ that satisfy
\[\norm{u}_{H^s_{per}}^2 \isdef \sum_{n \in \Z} (1 + n^2)^{s} \abs{\mathcal{F}u(n)}^2 < +\infty\,.\]
This is a Hilbert space, and we denote by $\inner{\cdot}{\cdot}_{H^s_{per}}$ its scalar product. For all $s \geq 0$, a function $u \in H^{s}_{per}$ is also interpreted as a distribution on the torus through the natural identification
\[\forall v \in C^\infty(\mathbb{T}_{2\pi}), \quad u(v) \isdef \frac{1}{2\pi}\int_{-\pi}^{\pi} u(\theta) v(\theta) d \theta\,.\]
This way, one can generalize the definition of $H^s_{per}$ to all real $s$ as the set of distributions $u$ such that the coefficients $\mathcal{F}u(n) \isdef u(\overline{e_n})$
satisfy ${\norm{u}_{H^s_{per}} < +\infty}$.
For all real $s$, $H^s_{per}$ can be decomposed into the direct sum $H^s_{per} = H^s_e \oplus H^s_o$
where $H^s_e$ (resp. $H^s_o$) is the set of even (resp. odd) functions in $H^s_{per}$. By definition, a distribution $u$ is even (resp. odd) if
\[\forall n \in \Z, \quad \mathcal{F}u(n) = \mathcal{F}u(-n) \quad \left(\textup{resp. } \mathcal{F}u(n) = - \mathcal{F}u(-n)\right)\,.\]
We take the notation $L^2_e \isdef H^0_e$ and $L^2_o = H^0_o$. Clearly, the families 
$${\left(n \mapsto \cos(n\theta)\right)_{n \geq 0}} \quad \textup{and} \quad {\left(n \mapsto \sin((n+1)\theta)\right)_{n \geq 0}}$$ 
provide Hilbert basis respectively of $L^2_e$ and $L^2_o$. We denote
\[H^{\infty}_{per} = \bigcap_{s \geq 0} H^{s}_{per}, \quad H^{-\infty}_{per} = \bigcup_{s \leq 0} H^{s}_{per}\]
and similarly for $H^{\infty}_e, H^{\infty}_o, H^{-\infty}_e$ and $H^{-\infty}_o$. 

\subsection{Definition of $T^s$ and $U^s$}

\label{sec:analyticalSetting}

Let $\omega$ be defined by 
\[\omega(x) = \sqrt{1 - x^2}\]
and let 
\[L^2_\frac{1}{\omega} = \enstq{u \in L^1_{loc}(-1,1)}{ \int_{-1}^1 \frac{\abs{u(x)}^2dx}{\sqrt{1 - x^2}} < + \infty}\,,\]
$$L^2_{\omega} \isdef \enstq{u \in L^1_\textup{loc}(-1,1)} {\int_{-1}^{1} {\abs{u(x)}^2}{\sqrt{1 - x^2} }dx< + \infty}.$$
Following the notations of \cite{mclean2000strongly}, we denote the Banach duality products of $L^2_\frac{1}{\omega}$ and $L^2_\omega$ respectively by $\duality{\cdot}{\cdot}_\frac{1}{\omega}$ and $\duality{\cdot}{\cdot}_\omega$ and the inner products respectively by $\inner{\cdot}{\cdot}_\frac{1}{\omega}$ and $\inner{\cdot}{\cdot}_\omega$. We take the normalization as
\[\inner{u}{v}_{\frac{1}{\omega}} = \duality{u}{\overline v}_\frac{1}{\omega} \isdef \frac{1}{\pi}\int_{-1}^{1} \frac{u(x) \overline{v(x)}}{\omega(x)}dx\,,\]
\[\inner{u}{v}_{\omega} = \duality{u}{\overline v}_{\omega} \isdef \frac{1}{\pi}\int_{-1}^{1} {u(x) \overline{v(x)}}{\omega(x)}dx\,.\]
We further denote $\norm{u}_{\frac{1}{\omega}}^2 \isdef\inner{u}{u}_{\frac{1}{\omega}}^2$ and $\norm{u}_{{\omega}}^2 \isdef\inner{u}{u}_{{\omega}}^2$. To $u \in L^2_\frac{1}{\omega}$ and $v \in L^2_\omega$, one can associate two functions $\mathcal{C}u$ and $\mathcal{S}v$ respectively in $L^2_e$ and $L^2_o$ by 
\[\mathcal{C}u(\theta) \isdef u(\cos\theta), \quad \mathcal{S}u(\theta) \isdef \sin\theta u(\cos\theta)\,.\]
Using the change of variables $x = \cos\theta$, one can check that the mappings 
\[\mathcal{C} : L^2_\frac{1}{\omega} \to L^2_{e}\,, \quad \mathcal{S} : L^2_{\omega} \to L^2_{o}\]
are isometric. Let us now introduce the Chebyshev polynomials of first and second kind (see e.g. \cite{mason2002chebyshev})
\[\forall x \in [-1,1], \quad T_n(x) = \cos(n\arccos(x)), \quad U_n(x) = \frac{\sin((n+1)\arccos(x))}{\sqrt{1 - x^2}}\,.\]
Both $T_n$ and $U_n$ are polynomials of degree $n$. Since 
\[\mathcal{C}T_n(\theta) = \cos(n\theta),\quad \mathcal{S}U_n(\theta) = \sin((n+1)\theta)\,,\]
the families $(T_n)_{n \geq 0}$ and $(U_n)_{n \geq 0}$ provide Hilbert basis of $L^2_\frac{1}{\omega}$ and $L^2_\omega$ respectively. Consequently, any functions $u \in L^2_\frac{1}{\omega}$, $v \in L^2_\omega$ can be expanded in Fourier-Chebyshev series of first and second kind respectively:
\[u = \sum_{n \geq 0} \hat{u}_n T_n \textup{ in } L^2_\frac{1}{\omega}\,, \quad v = \sum_{n \geq 0}{\check{v}_n} U_n \textup{ in } L^2_\omega \]
where the coefficients $\hat{u}_n$ and $\check{v_n}$ are obtained by orthogonal projection
\[\hat{u}_n = \frac{1}{\pi} \int_{-1}^{1} \frac{u(x) T_n(x)}{\omega(x)}dx, \quad \check{v}_n = \frac{1}{\pi}\int_{-1}^{1}v(x) U_n(x)\omega(x)\,dx\,.\]
Notice the difference in the accentuation for the coefficients in the series of first and second kind. One  has
\[\frac{1}{\pi} \int_{-1}^1 \frac{T_n(x) T_m(x)\,dx}{\omega(x)} = \begin{cases}
1 & \textup{if } n = m = 0\,, \\
\frac{1}{2} &\textup{if } n = m \neq 0\,,\\
0 & \textup{if } n \neq m\,,
\end{cases}\,\]
\[\frac{1}{\pi} \int_{-1}^1 U_n(x) U_m(x) \omega(x)\,dx = \begin{cases} \frac{1}{2} & \textup{if } n= m\,,\\
0 & \textup{if } n \neq m\,.
\end{cases}\]
The Parseval identity is transported to
\[\frac{1}{\pi}\int_{-1}^{1} \frac{\abs{u(x)}^2}{\omega(x)} = |\hat{u}_0|^2 + \frac{1}{2}\sum_{n = 1}^{+\infty} \abs{\hat{u}_n}^2, \quad \frac{1}{\pi}\int_{-1}^{1} {\abs{v(x)}^2}{\omega(x)} = \frac{1}{2}\sum_{n = 0}^{+ \infty} \abs{\check{v}_n}^2\,.\]

\begin{definition}
	For all $s \in \R$, we define $T^s$ as the set of formal series 
	\[u = \sum_{n = 0}^{+ \infty} \hat{u}_n T_n\]
	such that
	\[\norm{u}_{T^s}^2 \isdef \abs{\hat{u}_0}^2 + \frac{1}{2}\sum_{ n = 0}^{+ \infty} (1 + n^2)^s\abs{\hat{u}_n}^2 < + \infty\,.\]
	Similarly, $U^s$ is the set of formal series
	\[u = \sum_{n = 0}^{+ \infty} \check{u}_n U_n\]
	such that
	\[\norm{u}_{U^s}^2 \isdef \frac{1}{2}\sum_{ n = 0}^{+ \infty} (1 + n^2)^s\abs{\check{u}_n}^2 < + \infty\,.\]
	The scalar products 
	\[\inner{u}{v}_{T^s} \isdef \hat{u}_0 \conj{\hat{v}_0} + \frac{1}{2}\sum_{n \geq 0} (1 + n^2)^s \hat{u}_n \conj{\hat{v}_n}, \quad \inner{u}{v}_{U^s} \isdef \frac{1}{2}\sum_{n \geq 0} (1 + n^2)^s \check{v}_n \conj{\check{v}_n}\]
	endow $T^s$ and $U^s$ with a structure of Hilbert space for all $s$. For $s \geq 0$, the series defining elements of $T^s$ and $U^s$ are convergent in $L^2_\frac{1}{\omega}$ and $L^2_\omega$ respectively. Thus, $T^s$ and $U^s$ are naturally identified to subspaces of $L^2_\frac{1}{\omega}$ and $L^2_\omega$ respectively. Let $T^{\infty} = \cap_{s \in \R} T^s$ and $U^{\infty} = \cap_{s \in \R} U^s$. To $u \in T^s$, $v \in U^s$, one can associate the linear forms denoted by $\duality{u}{\cdot}_{\frac{1}{\omega}}$ and $\duality{v}{\cdot}_\omega$ and defined by
	\[\forall \varphi \in T^{\infty}, \quad \duality{u}{\varphi}_\frac{1}{\omega} \isdef \hat{u}_0 \hat{\varphi_0} + \frac{1}{2}\sum_{n = 0}^{+ \infty}\hat{u}_n\hat{\varphi}_n\,,\]
	\[\forall \varphi \in U^{\infty}, \quad \duality{v}{\varphi}_\omega \isdef \frac{1}{2}\sum_{n = 0}^{+ \infty}\check{v}_n \check{\varphi}_n\,.\]
	For $s \geq 0$, those linear forms coincide with
	\[\forall \varphi \in T^{\infty}, \quad \duality{u}{\varphi}_\frac{1}{\omega} \isdef \frac{1}{\pi}\int_{-1}^{1} \frac{u(x)\varphi(x)\,dx}{\omega(x)}\,,\]
	\[\forall \varphi \in U^{\infty}, \quad  \duality{v}{\varphi}_\omega = \frac{1}{\pi}\int_{-1}^{1}{v(x)\varphi(x)}{\omega(x)}\,dx \,\]
	justifying the notation. 
	For all $s$, the duals of $T^s$ and $U^s$ are the sets of linear forms
	\[\varphi \mapsto \duality{u}{\varphi}_\frac{1}{\omega}, \quad \varphi \mapsto \duality{v}{\varphi}_\omega\]
	respectively, where $u \in T^{-s}$, $v \in U^{-s}$. Finally, let 
	$$ T^{-\infty} = \displaystyle\bigcup_{s \in \R} T^s\,, \quad U^{-\infty} = \displaystyle\bigcup_{s \in \R} U^s\,.$$
\end{definition}
\noindent The spaces $T^s$ and $U^s$ thus introduced correspond to those of \cite{alougesAverseng}. 

\subsection{Basic properties}

Let $s_1,s_2 \in _R$, $\theta \in (0,1)$ and let $s = \theta s_1 + (1-\theta)s_2$. It is easy to check that
\[\forall u \in T^\infty, \quad \norm{u}_{T^s} \leq \norm{u}_{T^{s_1}}^\theta \norm{u}_{T^{s_2}}^{1 - \theta}\]
and 
\[\forall u \in U^\infty, \quad \norm{u}_{U^s} \leq \norm{u}_{U^{s_1}}^\theta \norm{u}_{U^{s_2}}^{1 - \theta}\,.\]
Therefore, 
\begin{lemma}
	$(T^s)_{s \in \R}$ and $(U^s)_{s \in \R}$ are exact interpolation scales. 
\end{lemma}
\begin{lemma}
	$C^{\infty}([-1,1])$ is dense in $T^s$ and $U^s$ for all $s$
\end{lemma}
\begin{proof}
	Any $u \in T^s$ (resp. $v \in U^s$) is the limit in $T^s$ (resp. $U^s$) of the sequence of polynomials
	\[u_N = \sum_{n = 0}^N \hat{u}_n T_n, \quad \textup{(resp. } v_N = \sum_{n = 0}^N \check{v}_n U_n)\,. \qedhere\]
\end{proof}
In view of the previous result, the maps $\mathcal{C}$ and $\mathcal{S}$ can be continuously extended respectively to the whole $T^{-\infty}$ and $U^{-\infty}$ by the definitions
\begin{equation}
\label{defCS}
\mathcal{C} T_n = (\theta \mapsto \cos(n\theta)), \quad \mathcal{S}U_n = (\theta \mapsto \sin((n+1) \theta))\,.
\end{equation}
\begin{lemma}
	\label{lemChar}
	For all $s$, $\mathcal{C}$ and $\mathcal{S}$ induce the following isomorphisms:
	\[\mathcal{C}: T^s \to H^s_e, \quad \mathcal{S}: U^s \to H^s_o\,.\]
\end{lemma}
\begin{corollary}
	For all $s' > s$, the following inclusions are compact:
	\[T^{s'} \subset T^s \quad \textup{and} \quad U^{s'} \subset U^s\,.\]
\end{corollary}

\noindent The polynomials $T_n$ and $U_n$ are connected by the formulas:
\begin{equation}
\label{TnAsUn}
T_0 = U_0, \quad T_1 = \frac{U_1}{2}, \quad \text{ and } \quad \forall n \geq 2, \quad T_n = \frac{1}{2}\left(U_n - U_{n-2}\right),
\end{equation}
\begin{equation}
\label{UnAsTn}
\forall n \in \N, \quad U_{2n} = 2\sum_{j = 0}^n T_{2j} - 1, \quad U_{2n+1} = 2\sum_{j=0}^n T_{2j+1}.
\end{equation}
Let us define the continuous maps $I : T^{-\infty} \to U^{-\infty}$ by 
\[\reallywidecheck{I \varphi}_0 = \hat{\varphi}_0 - \frac{\hat{\varphi}_2}{2}, \quad \reallywidecheck{I \varphi}_j = \frac{\hat{\varphi}_j - \hat{\varphi}_{j+2}}{2} \textup{ for } j \geq 1,\]
and $J: U^{\infty} \to T^{\infty}$ by
\[\widehat{J\varphi}_{0} = \sum_{n=0}^{+ \infty} \check{\varphi}_{2n}, \quad  \widehat{J\varphi}_j = \duality{\varphi}{T_j}_\frac{1}{\omega} =  2\sum_{n=0}^{+\infty} \check{\varphi}_{j + 2n} \textup{ for } j \geq 1.\]
One can show that $J$ is continuous for example using Hardy's inequality \cite{hardy1920note}: for all  $(u_n)_n \in l^2(\N)$, there holds
\[ \sum_{n = 1}^{+ \infty} \left(\sum_{k = n}^{+ \infty} \frac{u_k}{k}\right)^2 \leq C \sum_{k = 1}^{+ \infty} u_k^2\,.\]
In view of eqs.~\eqref{TnAsUn} and \eqref{UnAsTn}, $I$ and $J$ should be seen as identification mappings. Accordingly, for $u \in T^s$ and $v \in U^s$, we write $u = v$ if $Iu = v$ or $u = Jv$. The continuities of $I$ and $J$ then express the fact that there hold continuous inclusions $T^{-\infty} \subset U^{-\infty}$ and $U^{\infty} \subset T^{\infty}$. The spaces $(T^s)_s$ and $(U^s)_s$ are interlaced as follows:
\begin{lemma}
	\label{inclusionsTsUs}
	There hold the following continuous inclusions:
	\[\forall s \in \R, \quad T^s \subset U^s\,,\] 
	\[\forall s \geq 1, \quad U^s \subset T^{s-1}\,.\]
\end{lemma}
\begin{proof}
	The continuity of $I$ from $T^s$ to $U^s$ is immediate, implying the first inclusion. For the second one, let $s\geq1$. By a density argument, it suffices to show that for all $u \in U^{\infty}$,
	\[\norm{Ju}_{T^{s-1}} \leq C \norm{u}_{U^s}\,.\]
	Let $u \in U^{\infty}$. One has
	\begin{eqnarray*}
		\norm{Ju}_{T^{s-1}}^2 &\leq&  \sum_{n=0}^{+\infty}(1+n^2)^{s-1} \abs{ \widehat{Ju}_n}^2 \\
		&\leq&  C\sum_{n=0}^{+\infty}(1+n^2)^{s-1} \left(\sum_{k=n}^{+\infty} \abs{\check{u}_k}\right)^2\\
		&\leq& C\sum_{n=0}^{+ \infty}\left(\sum_{k=n}^{+\infty}(1+k^2)^{\frac{s-1}{2}} \abs{\check{u}_k})\right)^2\,.
	\end{eqnarray*}	
	Notice that the last step is only possible if $s \geq 1$. 
	Furthermore, the $U^s$ norm of $u$ is equal (up to a multiplicative factor) to the $l^2$ norm of the sequence $\left((1+n^2)^{s/2} \abs{\check{u}_n}\right)_{n \geq 1}$. By the Hardy's inequality stated above, the sequence $(r_n)_n$ defined by 
	\[\forall n \geq 0, \quad r_n \isdef \sum_{k=n}^{+ \infty} (1+k^2)^{\frac{s-1}{2}} \abs{\check{u}_k}\] 
	thus lies in $l^2(\N)$ with a $l^2$ norm controlled by $\norm{u}_{U^s}$. Combining this with the previous inequalities provides the announced estimate
	\[\norm{Ju}_{T^{s-1}}^2 \leq C \norm{u}_{U^s}^2\,,\]
	from which the result follows. 
\end{proof}
\noindent The next results give some precision on the case $s < 1$.
\begin{lemma}
	\label{weakInclusions}
	For $s > \frac{1}{2}$, and for any $\varepsilon > 0$, 
	\[\forall s>\frac{1}{2}, \quad U^s \subset T^{s - 1 - \varepsilon}\,.\]
	Finally, $U^{\frac{1}{2}}$ is not continuously embedded in $T^{s}$ for any $s \in \R$. 
\end{lemma}
\begin{proof}
	Let $s > \frac{1}{2}$ and let $u \in U^s$. Using Cauchy-Schwarz's inequality, one has
	\[\left(\sum_{k = n}^{+\infty} \abs{\check{u}_k}\right)^2 \leq \norm{u}_{U^s}^2 \sum_{k = n}^{+\infty} \frac{1}{(1+ n^2)^{s}}\,.\]
	For the second term of the right hand side there holds the classical estimate
	\[\sum_{k = n}^{\infty} \frac{1}{(1 + n^2)^s} \leq C(1 + n^2)^{-s + \frac{1}{2}}\,,\]
	Let $\sigma < s - 1$. We thus have 
	\[\norm{Ju}_{T^\sigma} \leq C\sum_{n = 0}^{+\infty} (1 + n^2)^{\sigma - s + \frac{1}{2}} \norm{u}_{U^s}^2\,.\]
	Since $\sigma - s > 0$, the infinite sum converges and this proves the inclusion $U^s \subset T^{\sigma}$. 
	For the second statement, let
	\[u = \sum_{n = 0}^{+\infty} \frac{1}{n \ln n} U_n\,.\]
	One can check that $u \in U^{\frac{1}{2}}$. Let us assume by contradiction that there is a continuous inclusion $U^{\frac{1}{2}} \subset T^{s}$ for some $s$. The sequence of polynomials
	\[u_N = \sum_{n = 0}^{N} \check{u}_n U_n\]
	converges to $u$ in $U^{\frac{1}{2}}$. By continuity of the inclusion $U^\frac{1}{2} \to T^{s}$,  the sequence  $(\duality{u_N}{T_0}_\frac{1}{\omega})_{N\in \N}$ must converge with limit $\duality{u}{T_0}_\frac{1}{\omega}$. But
	\[\duality{u_N}{T_0}_\frac{1}{\omega} = \sum_{n=0}^{N} \check{u}_n\duality{U_n}{T_0}_\frac{1}{\omega} = \sum_{k = 0}^{\lfloor \frac{N}{2} \rfloor} \frac{1}{2k \ln(2k)}\,, \]
	where, in the last equality, we have used the identity
	\[\duality{U_n}{T_0}_\frac{1}{\omega} = \begin{cases}
	1 &\text{ if } n \text{ is even },\\
	0 &\text{ otherwise }.
	\end{cases}\]
	This can be checked for example using eq.  \eqref{UnAsTn}. 
	The last sum diverges to $+\infty$ when $N$ goes to infinity, giving the contradiction. 
\end{proof}
\noindent As a corollary of \autoref{inclusionsTsUs}, we have:
\begin{corollary}
	\label{corTinfUinf}
	\[T^{\infty} = U^{\infty}\,.\]
\end{corollary}

\subsection{Regularity properties}
We now investigate some regularity properties of the elements of $T^s$ and $U^s$. 
\begin{lemma}
	\label{LemInjectionsContinues}
	For all $\varepsilon >0$, if $u \in T^{\frac{1}{2} + \varepsilon}$, then $u$ is continuous and
	\[ \exists C : \forall x \in [-1,1], \quad \abs{u(x)} \leq C \norm{u}_{T^{1/2 + \varepsilon}}.\]	
	Similarly, if $u \in U^{3/2 + \varepsilon}$, then $u$ is continuous and 
	\[ \exists C : \forall x \in [-1,1], \quad \abs{u(x)} \leq C \norm{u}_{U^{3/2 + \varepsilon}}.\]
\end{lemma}
\begin{proof}
	Let $u \in T^{1/2 + \varepsilon}$. Then $u = \varphi(\arccos(x))$ where $\varphi \in H^{1/2 + \varepsilon}_{per}$. The first statement  follows from the continuity of $\arccos$ and the Sobolev embedding in $H^{\frac{1}{2}+\varepsilon}_{per} \subset C^0$. 
	The second statement is deduced from the first and the continuous inclusion $U^{s} \subset T^{s-1 - \varepsilon}$ proved in \autoref{weakInclusions}. 
\end{proof}	
Let us now introduce some differential operators. Let $\partial_x$ be the derivation operator and $\omega$ the operator $u(x) \mapsto \omega(x)u(x)$ with $\omega(x) = \sqrt{1 - x^2}$, $T_n$ and $U_n$ satisfy the following identities:
\begin{eqnarray}
-(\omega\partial_x)^2 T_n &=& n^2T_n\,, \label{cheb1}\\
-(\partial_x\omega)^2 U_n &=& (n+1)^2U_n\, .\label{cheb2}
\end{eqnarray}
These are just the differential equations defining $T_n$ and $U_n$  written in divergence form. Notice that here and in the following, $\partial_x\omega$ denotes the composition of operators $\partial_x$ and $\omega$ and not the function $x \mapsto \partial_x\omega(x)$. One can also check the identities 
\begin{align}
\partial_x T_n &= n U_{n-1}\,, \label{der1} \\
-\omega \partial_x \omega U_n &= (n+1)T_{n+1}\,. \label{der2}
\end{align}
The first one is obtained for example from the definition of $T_n$, from which we deduce the second one after using $-(\omega \partial_x)^2 T_{n+1} = (n+1)^2 T_{n+1}$. 
\begin{lemma}
	\label{LemTinfCinf}
	There holds
	\[T^{\infty} = U^\infty = C^{\infty}([-1,1])\,.\]
\end{lemma}
\begin{proof}
	Recall that $T^\infty = U^{\infty}$ ( \autoref{corTinfUinf}).	Let $u \in C^{\infty}([-1,1])$, then we can obtain by induction using integration by parts and \eqref{cheb1}, that for any $k \in \N$
	\[\hat{u}_n = \frac{(-1)^k}{n^{2k}} \int_{-1}^{1} \dfrac{(\omega\partial_x)^{2k} u(x) T_n(x)}{\omega(x)}dx.\]
	Noting that $(\omega \partial_x)^2 = (1-x^2)\partial_x^2 - x \partial_ x$, the function $(\omega \partial_x)^{2k}u$ is $C^{\infty}$, and since $\norm{T_n}_\infty = 1$, the integral is bounded independently of $n$. Thus, the coefficients $\hat{u}_n$ have a fast decay, proving that $C^{\infty}([-1,1]) \subset T^{\infty}$. 
	
	For the converse inclusion, if $u \in T^{\infty}$, the series
	\[ u(x) = \sum_{n=0} \hat{u}_n T_n(x)\]
	is normally converging since $\norm{T_n}_\infty = 1$, so that $u$ is a continuous function. This proves $T^{\infty} \subset C^0([-1,1])$. It now suffices to show that $\partial_x u \in T^{\infty}$ and apply an induction argument. Applying term by term differentiation, since $\partial_x T_n = n U_{n-1}$ for all $n$ (with the convention $U_{-1} = 0$),
	\[\partial_x u(x) = \sum_{n=1}^{+\infty} n \hat{u}_n U_{n-1}(x).\] 
	Therefore, $\partial_x u$ is in $U^{\infty} = T^{\infty}$ which proves the result. 
\end{proof}

We now extend the definition of the differential operators $\partial_x$ and $\omega\partial_x\omega$ appearing in eqs (\ref{der1}) and  (\ref{der2}).

\begin{lemma}
	\label{derivations}
	For all real $s$, the operator $\partial_x$ can be extended into a continuous map from $T^{s+1}$ to $U^{s}$ defined by 
	\[\forall v \in \Cinf([-1,1]), \quad \duality{\partial_x u}{v}_{\omega} \isdef -\duality{u}{\omega \partial_x \omega v}_{\frac{1}{\omega}} \,.\] 
	In a similar fashion, the  operator $\omega \partial_x \omega$ can be extended into a continuous map from $U^{s+1}$ to $T^{s}$ defined by
	\[\forall v \in \Cinf([-1,1]), \quad \duality{\omega \partial_x \omega u}{v}_\frac{1}{\omega} \isdef -\duality{u}{\partial_x v}_\omega.\]
\end{lemma}
\begin{proof}
	Using eqs (\ref{der1}) and (\ref{der2}), one can check that the formulas indeed extend the usual definition of both operators for smooth functions. We now show that the map $\partial_x$ extended this way is continuous from $T^{s+1}$ to $U^s$. The definition 
	\[\forall v \in U^{\infty}, \quad \duality{\partial_x u}{v}_\omega \isdef -\duality{u}{\omega \partial_x \omega v}_\frac{1}{\omega}\]
	gives a sense to $\partial_x u$ for all $u$ in $T^{-\infty}$, as a duality $T^{-\infty} \times T^{\infty}$ product, because if $v \in U^{\infty} (= C^{\infty}([-1,1])$, then $\omega \partial_x \omega v = (1-x^2)v' - xv$ also lies in $C^{\infty}([-1,1]) (= T^\infty)$. Letting $w = \partial_x u$, we have by definition for all $n$
	\[\check{w}_n = \duality{w}{U_n}_{\omega} = - \duality{u}{\omega \partial_x \omega U_n}_\frac{1}{\omega} = n \duality{u}{T_{n+1}}_\frac{1}{\omega} = n\hat{u}_{n+1}\,.\]
	This implies the announced continuity with
	\[ \norm{w}_{U^s} \leq \norm{u}_{T^{s+1}}\,.\]
	The properties of $\omega \partial_x \omega$ on $T^s$ are established similarly. 
\end{proof}

\subsection{Generalization to a curve}

All of the previous analysis can be generalized to define two families of spaces $T^s(\Gamma)$ and $U^s(\Gamma)$ of functions defined on a smooth curve $\Gamma$ by means of a $C^\infty$ diffeomorphism.

\label{TsUs(Gamma)}

\subsubsection*{Parametrization of the curve}
The notation of this paragraph will be used at several points in the remainder of this work. Let $\Gamma$ be a smooth open simple curve in $\mathbb{R}^2$ parameterized by a $C^\infty$ diffeomorphism $r : [-1,1] \to \Gamma$. We assume that $\abs{r'(x)} = \frac{\abs{\Gamma}}{2}$ for all $x\in [-1,1]$, where $\abs{\Gamma}$ is the length of $\Gamma$. 
Let $\opFromTo{R_\Gamma}{\Cinf(\overline  \Gamma)}{\Cinf([-1,1])}$ be the pullback defined by 
\[\forall x \in [-1,1], \quad R_\Gamma u(x) = u(r(x))\,.\]
The tangent and normal vectors on the curve, $\tau$ and $n$, are respectively defined by 
\[\forall x \in [-1,1], \quad\tau(x) = \frac{\partial_x r(x)}{\abs{\partial_x r(x)}}, \quad n(x) = \frac{\partial_x \tau(x)}{\abs{\partial_x \tau'(x)}}\,.\]
Let $N : \Gamma \to \R^2$ be such that $N(r(x)) = n(x)$, that is, $N = R_{\Gamma}^{-1} n$. Let $\kappa(x)$ be the signed curvature of $\Gamma$ at the point $r(x)$. Frenet-Serret's formulas give
\begin{eqnarray*}
	\forall x,y \in [-1,1], \quad r(y) & = & r(x) + (y-x) \frac{\abs{\Gamma}}{2} \tau(x) + \frac{(y-x)^2}{2} \frac{\abs{\Gamma}^2}{4} \kappa(x) n(x)\\ 
	&& +\frac{(x-y)^3}{6} \frac{\abs{\Gamma}^3}{8} (\kappa'(x) n(x) - \kappa(x)^2 \tau(x)) + O\left((x-y)^4\right)\,,
\end{eqnarray*}
so that
\begin{equation}
\forall x,y \in [-1,1], \quad \abs{r(x) - r(y)}^2 = \frac{\abs{\Gamma}^2}{4}(y-x)^2 - \frac{(y-x)^4}{192}\abs{\Gamma}^4 \kappa(x)^2 + O(x-y)^5\,. 
\label{expansion_r}
\end{equation}
For $u,v \in L^2(\Gamma)$, we have by change of variables in the integral
\[\duality{u}{v}_{L^2(\Gamma)} = \frac{\abs{\Gamma}}{2}\duality{R_\Gamma u}{R_\Gamma v}_{L^2(-1,1)} \,.\]
The tangential derivative $\partial_\tau$ on $\Gamma$ satisfies 
\begin{equation}
\partial_\tau = \frac{2}{\abs{\Gamma}}R_{\Gamma}^{-1}\partial_x R_{\Gamma}\,.
\label{param1}
\end{equation}
We also define a ``weight" operator on the curve as
\begin{equation}
\omega_\Gamma \isdef \frac{\abs{\Gamma}}{2} R_{\Gamma}^{-1}  \omega R_{\Gamma}\,.
\label{param2}
\end{equation}
Finally, the uniform measure on $\Gamma$ is denoted by $d\sigma$.

\subsubsection*{Spaces \texorpdfstring{$T^s(\Gamma)$ and $U^s(\Gamma)$}{Ts(Gamma) and Us(Gamma)}}

The definition of the spaces $T^s$ can be transported on the curve $\Gamma$, replacing the basis $(T_n)_n$ and $(U_n)_n$ by $(R_{\Gamma}^{-1}T_n)_n$ and $(R_{\Gamma}^{-1}U_n)_n$. The spaces $T^s(\Gamma)$ and $U^s(\Gamma)$ are thus defined as the sets of formal series respectively of the form 
\[u = \sum_{n \in \N} \hat{u}_n R_{\Gamma}^{-1}T_n\,, \quad v = \sum_{n \in \N} \check{v}_n R_{\Gamma}^{-1}U_n\,,\]
where $R_\Gamma u = \sum\hat{u}_n T_n \in T^s$ and $R_\Gamma v = \sum\check{v}_n U_n  \in U^s$. To $u$ and $v$ are associated the linear forms 
\[\forall \varphi \in \Cinf(\overline{\Gamma}), \quad \duality{u}{\varphi}_\frac{1}{\omega_\Gamma} \isdef \duality{R_\Gamma u}{R_{\Gamma}\varphi}_\frac{1}{\omega}\,,\]
\[\forall \varphi \in \Cinf(\overline{\Gamma}), \quad \duality{v}{\varphi}_{\omega_\Gamma} \isdef \frac{\abs{\Gamma}^2}{4}\duality{R_\Gamma v}{R_{\Gamma}\varphi}_{\omega}\,.\] 
\noindent The results of the previous section are easily extended to this new setting:
\begin{lemma} For all $s \in \R$, $T^s(\Gamma)$ and $U^s(\Gamma)$ are Hilbert spaces for the scalar products 
	\[\inner{u}{v}_{T^s(\Gamma)} = \inner{R_\Gamma u}{R_\Gamma v}_{T^s}\,,\]
	\[\inner{u}{v}_{U^s(\Gamma)} = \frac{\abs{\Gamma}^2}{2}\inner{R_\Gamma u}{R_\Gamma v}_{U^s}\,.\]
	With these definitions, 
	\[\inner{u}{v}_{T^0(\Gamma)} = \duality{u}{\overline{v}}_\frac{1}{\omega_\Gamma} = \int_{\Gamma} \frac{u(x) \overline{v(x)}}{ \omega_\Gamma(x)} dx\,,\]
	\[\inner{u}{v}_{U^0(\Gamma)} = \duality{u}{\overline{v}}_{\omega_\Gamma}  = \int_{\Gamma} \omega_\Gamma(x) u(x) \overline{v(x)} dx\,.\]
	In particular $T^0(\Gamma) = L^2_\frac{1}{\omega_\Gamma}$ and $U^0(\Gamma) = L^2_{\omega_\Gamma}$. For $s \in \R$, the dual of $T^s(\Gamma)$ (resp. $U^s(\Gamma)$) is the set of linear forms $\duality{u}{\cdot}_\frac{1}{\omega_\Gamma}$ (resp. $\duality{v}{\cdot}_{\omega_\Gamma}$) such that $u \in T^{-s}$ (resp. $v \in U^{-s}$). For $s < s'$, the injections $T^{s'}(\Gamma) \subset T^s(\Gamma)$ and $U^{s'}(\Gamma) \subset U^s(\Gamma)$ are compact. Furthermore, $(T^s(\Gamma))_{s\in \R}$ and $(U^s(\Gamma))_{s\in \R}$ are two Hilbert exact interpolation scales. For all $s \in \R$, $T^s(\Gamma) \subset U^s(\Gamma)$ and for all $s \geq 1$, $U^s(\Gamma) \subset T^{s-1}(\Gamma)$ with continuous inclusions. For $\varepsilon > 0$, $T^{1/2 + \varepsilon}(\Gamma) \subset C^0(\Gamma)$ and $U^{3/2+\varepsilon} \subset C^0(\Gamma)$ with continuous inclusions. Finally, $T^\infty(\Gamma) = U^\infty(\Gamma) = C^\infty(\overline{\Gamma})$ and this space is dense in $T^{s}(\Gamma)$ and $U^s(\Gamma)$ for all $s\in \R$.  
\end{lemma}

\section{Pseudo-differential operators on open curves}

\label{sec:PDO}

We now introduce the two classes of pseudo-differential operators on open curves. Our approach can be summarized as follows. Through the change of variables $x = \cos \theta$, an operator on the segment can be viewed as an operator on the torus $\mathbb{T}_{2\pi}$. On this geometry, a simple algebra of pseudo-differential operators exists \cite{turunen1998symbol}. However, the inverse change of variables, $\theta = \arccos(x)$ has singularities at $x = -1$ and $x = 1$, preventing the simple transfer of the properties of this algebra back to the segment. To solve this problem, we consider pseudo-differential operators which preserve even or odd functions. The symbols of those operators have some parity properties that lead to cancellations of the singularity. As a result, two classes of operators emerge, related to pseudo-differential operators on the torus that stabilize even and odd functions respectively. 

We start by collecting some facts on periodic pseudo-differential operators in section \ref{subsec:PPDO}. We then introduce a first class of pseudo-differential operators on the segment (and more generally on smooth open curves) in section \ref{subsec:PDOTs}, which is based on the scales of Hilbert spaces $(T^s)_{s\in \R}$ presented in the previous section. We show that the usual properties of pseudo-differential operators hold in this class. The pseudo-differential operators based on $(U^s)_{s\in \R}$ are introduced in section \ref{subsec:PDOUs}. Some links between the two classes are drawn in \autoref{subsec:linkTsUs} and finally, we state some results about square-roots of classical elliptic pseudo-differential operators in \autoref{subsec:squareRoot}.

\subsection{Periodic pseudo-differential operators}

\label{subsec:PPDO}

On the family of periodic Sobolev spaces $H^s_{per}$, a class of periodic pseudo differential operators (PPDO) is studied in \cite{turunen1998symbol}. We briefly reproduce here the material needed for our purposes. A PPDO of order $\alpha$ on $H^s$ is an operator of the form 
\[Au(\theta) =  \sum_{n \in \Z} \sigma_A(\theta,n) \mathcal{F}u(n) e^{in\theta}\,\]
for a ``prolongated symbol" $\sigma_A \in C^{\infty}(\mathbb{T}_{2\pi} \times \R)$ satisfying 
\begin{equation}
\label{SymbolsCond}
\forall j,k \in \N, \quad \exists C_{j,k} >0 : \quad \abs{D^j_\theta D_\xi^k \sigma_A(\theta,\xi)} \leq C_{j,k}(1 + \abs{\xi})^{\alpha - k}\,,
\end{equation}
where
$${D_\theta \isdef \frac{1}{i}\frac{\partial}{\partial \theta}}, \quad  D_\xi \isdef \frac{1}{i}\frac{\partial}{\partial \xi}\,.$$ 
The class of symbols that satisfy \eqref{SymbolsCond} is denoted by $\Sigma^\alpha$. Let $\Sigma^{\infty} \isdef \cup_{\alpha \in \R} \Sigma^\alpha$ and $\Sigma^{-\infty} = \cap_{\alpha \in \R} \Sigma^\alpha$. The operator defined by a symbol $\sigma$ is denoted by $\textit{Op}(\sigma)$ and the set of PPDOs of order $\alpha$ is denoted by $\textit{Op}(\Sigma^\alpha)$. The PPDOs of $\textit{Op}(\Sigma^{-\infty})$ are called smoothing operators. 

The prolongated symbol is not unique but determined uniquely at integer values of $\xi$ by (see \cite{turunen1998symbol}):
\begin{equation}
\label{symbolUniqueDetermine}
\forall (\theta,n) \in \mathbb{T}_{2\pi} \times \N, \quad \sigma_A(\theta,n) =  \conj{e_{n}(\theta)}Ae_n(\theta)\,,
\end{equation}
where we recall the notation $e_n(\theta) = e^{in\theta}$. This justifies the terminology of ``prolongated symbol". The operator $A$ is in $\textit{Op}(\Sigma^\alpha)$ if and only if
\begin{equation}
\forall j,k \in \N, \quad  \exists C_{j,k} > 0 :\quad  \abs{D_\theta^j \Delta_n^k \sigma_A(\theta,n)} \leq C_{j,k}(1 + \abs{n})^{\alpha - k}\,,
\label{condPPDO}
\end{equation}
where $\Delta_n \phi(\theta,n) = \phi(\theta,n+1) - \phi(\theta,n)$. That is, if the symbol defined in \eqref{symbolUniqueDetermine} satisfies \eqref{condPPDO}, then there exists a prolongated symbol satisfying \eqref{SymbolsCond}. Because of this, we write $\sigma \in \Sigma^p$ for a symbol $\sigma(\theta,n)$ that can be prolongated to a symbol ${\sigma}(\theta,\xi) \in \Sigma^p$. An operator in $\textit{Op}(\Sigma^\alpha)$ maps $H^s$ to $H^{s - \alpha}$ continuously for all $s \in \R$. The composition of two operators in $\textit{Op}(\Sigma^\alpha)$ and $\textit{Op}(\Sigma^\beta)$ gives rise to an operator in $\textit{Op}(\Sigma^{\alpha+\beta})$. 
If two symbols $a$ and $b$ in $\Sigma^{ \infty}$ satisfy $a - b  \in \Sigma^{\alpha}$, we write $a = b + \Sigma^\alpha$. 
\begin{definition}
	Let $a \in \Sigma^{\infty}$. If there exists a sequence of reals $(p_j)_{j \in \N}$ such that $p_{j+1} < p_{j}$ and a sequence of symbols $a_j \in \Sigma^{p_j}$ such that for all $N$, $a = \displaystyle\sum_{i = 0}^{N}a_i + \displaystyle\Sigma^{p_{N+1}}$, we write  $a \sim \displaystyle\sum_{i = 0}^{+ \infty} a_i \,.$
	This is called an asymptotic expansion of the symbol $a$. 
\end{definition}
The symbol of the composition of two PPDOs $A$ and $B$ is denoted by $\sigma_A \# \sigma_B$ and satisfies the asymptotic expansion \cite{turunen1998symbol}
\begin{equation}
\label{a_diese_b}
\sigma_A \# \sigma_B(\theta,\xi) \sim \sum_{j = 0}^{+\infty}\frac{1}{j!} \left(\frac{\partial}{\partial \xi}\right)^j \sigma_A(\theta,\xi) D_\theta^j \sigma_B(\theta,\xi)\,.
\end{equation}
In particular, if $A \in Op(\Sigma^\alpha)$ and $B \in Op(\Sigma^\beta)$, then
\[AB - BA \in Op(\Sigma^{\alpha + \beta - 1})\,.\]
\begin{definition}
	A symbol $a \in \Sigma^\alpha$ is ``classical" if it admits an asymptotic expansion of the form
	\[a \sim \sum_{i = 0}^{+\infty} a_{\alpha - i}\] 
	where the symbols $a_{\alpha - i}$ are positive homogeneous of order $\alpha - i$ for $|\xi| \geq 1$, i.e.
	\[\forall |\xi| \geq 1, \forall \lambda > 0, \quad  \alpha_i(x,\lambda \xi)  = \lambda^{\alpha - i} \alpha_i(x,\xi)\,.\]
	In this case, the symbol $a_\alpha$ is called the principal symbol of $a$. A symbol $\sigma$ is said to be elliptic if it satisfies
	\[\exists (C,M) :  \forall |n| \geq M, \quad |\sigma(\theta,n)| \geq C (1 + |n|)^{\alpha}\,.\] 
	A classical symbol is elliptic if and only if its principal symbol does not vanish
	\[\forall \xi \neq 0, \quad a_\alpha(x,\xi) \neq 0\,.\]
	A PPDO is said to be classical (resp. elliptic) if it admits a classical (resp. elliptic) symbol. 
\end{definition}
A standard result in pseudo-differential theory is that elliptic operators can be inverted modulo smoothing operators:
\begin{proposition}[{See \cite[Thm 4.5]{ruzhansky2010quantization}}]
	\label{parametrixElliptic}
	Let $A$ be an elliptic PPDO of order $\alpha$. Then there exists an elliptic PPDO $B$ of order $-\alpha$ such that
	\[BA = I + R_1, \quad AB = I + R_2\]
	where $R_1, R_2$ are smoothing operators. 
	The operator $B$ is called a parametrix of $A$. If $A$ is classical, then it admits a classical parametrix.
\end{proposition}
\begin{corollary}
	\label{SpectralTheoremPPDO}
	If $A$ is an elliptic PPDO, then
	\[\forall u \in H^{-\infty}_{per}\quad u \in H^\infty_{per} \iff Au \in H^{\infty}_{per}\,.\] 
	If $A$ is of order $\alpha \neq 0$, it admits a family of eigenfunctions $(u_n)_{n \in \N}$ that form a complete orthogonal basis of $L^2(\mathbb{T}_{2\pi})$. If $\alpha > 0$, the functions $u_n$ are $C^{\infty}$ and the eigenvalues can be chosen in increasing order diverging to $+\infty$.
\end{corollary}
\begin{proof}
	Let $A \in Op(\Sigma^{\alpha})$. The direct implication of the first statement is a consequence of the continuity of $A$ from $H^s$ to $H^{s + \alpha}$. For the reciprocal statement, let $B$ be a parametrix of $A$ and let $R$ be the smoothing operator $R = BA - I_d$. We have
	\[u = BAu - Ru\,.\] 
	Since $Au$ is smooth, so is $BAu$ by the direct implication. Moreover, $Ru$ is always smooth since $R$ is a smoothing operator. This proves the first claim. If $\alpha < 0$, a complete orthogonal basis of eigenvectors is provided by the spectral theorem, because in this case, $A$ is compact in $L^2(\mathbb{T}_{2\pi})$. For the case $\alpha > 0$, the previous result implies that $A$ is Fredholm of index $0$ and thus has a compact resolvent. It remains to show that the eigenvectors are smooth. Fix $u$ such that 
	\[Au = \lambda u\,.\]
	Then, left-multiplying by the parametrix $B$ of $A$, we have
	\[u = \lambda Bu - Ru\]
	Since $B : H^{s} \to H^{s-\alpha}$ for all $s \in \R$ and since $R$ is smoothing, a simple bootstrap argument shows that $u \in H^{\infty}_{per}$.  
\end{proof}
\begin{proposition}[{see \cite{turunen1998symbol}}]	
	\label{thrunen}
	Consider an integral operator $K$ of the form 
	\[K : u \mapsto \frac{1}{2\pi}\int_{-\pi}^\pi a(\theta,\theta') h(\theta-\theta') u(\theta') d\theta'\,,\]
	where $a$ is $2\pi$-periodic and $C^{\infty}$ in both arguments and $h$ is a $2\pi$-periodic distribution. Assume that the Fourier coefficients $\mathcal{F}h(n)$ of $h$ can be prolonged to a function $\hat{h}(\xi)$ on $\R$ such that
	\[\forall k \in \N, \quad \exists C_k > 0: \quad \abs{\partial^k_\xi \hat{h}(\xi)} \leq C_k (1 + \abs{\xi})^{\alpha - k}\,\]
	for some $\alpha$. Then $K$ is in $\textit{Op}(\Sigma^\alpha)$ with a symbol satisfying the asymptotic expansion
	\begin{equation}
	\label{FormuleIntegralOperatorSymbol}
	\sigma_K(\theta, \xi) \sim \sum_{j = 0}^{+ \infty} \frac{1}{j!} \left(\frac{\partial}{\partial \xi}\right)^j \hat{h}(\xi) D_{t}^ja(t,\theta)_{| t= \theta}\,.
	\end{equation}
\end{proposition}
\noindent In particular, taking $h \equiv 1$, we see that for any functions $a \in C^{\infty}(\mathbb{T}_{2\pi}^2)$, the operator 
\[K : u \mapsto \frac{1}{2\pi}\int_{-\pi}^{\pi} a(\theta,\theta') u(\theta') d\theta'\]
is smoothing. 

\subsection{Pseudo-differential operators on \texorpdfstring{$T^s(\Gamma)$}{Ts(Gamma)}}

\label{subsec:PDOTs}
\begin{definition}
	\label{DefOpTs}
	Let $A$ be an operator on $T^{-\infty}$ and assume that there exists a couple of functions $a_1$ and $a_2$ defined on $[-1,1]\times \N$, that are $\Cinf$ in the first variable and such that for all $n \in \N$,
	\begin{equation}
	AT_n = a_1(x,n) T_n - \omega^2 a_2(x,n) U_{n-1}\,,
	\label{defOpTs}
	\end{equation}
	with, by convention, $U_{-1} = 0$. The operator defined by the previous formula is denoted by $\textit{Op}_T(a_1,a_2)$. Define the symbol $\tilde{\sigma}(a_1,a_2)$ on $\mathbb{T}_{2\pi} \times \Z$ by
	\[\tilde{\sigma}(a_1,a_2)(\theta,n) \isdef a_1(\cos\theta,\abs{n}) + i \sin\theta\,\textup{sign}(n) a_2(\cos\theta,\abs{n})\,.\]
	We say that $(a_1,a_2) \in S^\alpha_T$ if $\tilde{\sigma}(a_1,a_2) \in \Sigma^\alpha$.  In this case, we say that $A$ is a pseudo-differential operator on $T^s$ and that the couple of functions $(a_1,a_2)$ is a pair of symbols for $A$. We denote $S^{\infty}_T \isdef \cup_{\alpha \in \R} S^\alpha_T$ and $S^{-\infty}_T = \cap_{\alpha \in \R} S^\alpha_T$.  The set of pseudo-differential operators (of order $\alpha$) in $T^{-\infty}$ is denoted by $\textit{Op}(S^\infty_T)$ (by $\textit{Op}(S^\alpha_T)$). The operator $A$ is said to be elliptic if it admits a pair of symbols $(a_1,a_2)$ such that $\tilde{\sigma}(a_1,a_2)$ is elliptic. Finally, if $A, B \in Op(S^{\infty}_T)$ are such that $A - B \in Op(S^{\alpha}_T)$, we write
	\[A = B + T_{\alpha}\,.\]
\end{definition} 
\begin{remark}
	It is easy to construct non-trivial symbols in $S^{-\infty}_T$ for the null operator. For example, for some $m \in \N$, take $a(x,n) = \delta_{n = m} \omega^2(x) U_{m-1}(x)$ and $b(x,n) = \delta_{n = m} T_m(x)$. Because of this, a pair of symbol for a pseudo-differential operator on $T^s$ is not unique. 
\end{remark}
\subsubsection*{Examples}
\begin{itemize}
	\item[(i)] Recall that the operator $(\omega \partial_x)^2$ satisfies
	\[-(\omega \partial_x)^2 T_n = n^2 T_n\,.\]
	Therefore, $(\omega \partial_x)^2$ admits the pair of symbols $a_1(x,n) = -n^2$, $a_2(x,n) = 0$. We have
	\[\tilde{\sigma}(a_1,a_2)(\theta,n) = n^2 \in \Sigma^2\,,\]
	thus $(\omega \partial_x)^2 \in S^2_T$ by definition, and this operator is elliptic. 
	\item[(ii)] Similarly one can check that for all $\psi \in C^\infty[-1,1]$, the operator $u \mapsto \psi u$ is in $Op(S_T^0)$.
	\item[(iii)] Let $x$ denote the operator multiplication by $x$ and let $A = (\omega \partial_x)^2 x$. Using the identities
	\[xT_n = \frac{T_{n+1} + T_{n-1}}{2}, \quad \omega^2 U_n = \frac{T_{n-1} - T_{n+1}}{2}\,,\]
	one can check that $A$ admits the pair of symbols
	\[a_1(x,n) = n^2(x + 1), \quad a_2(x,n) = -2n\omega^2\]
	and is in $Op(S_T^{2})$. Notice that it is not possible to find a pair of symbols of $A$ with $a_2 = 0$. This second part in the symbol is thus necessary to allow $Op(S^\infty_T)$ to be an algebra. More generally, we shall see below how the pair of symbols of a composition $AB$ can be systematically obtained from the pair of symbols of $A$ and $B$ using symbolic calculus.
	\item[(iv)] We will see in the next section that the most simple operator of $Op(S^{-1}_T)$, given by
	\[AT_n = \frac{1}{n}T_n\,,\]
	is closely related to the Laplace weighted single-layer potential on a segment.
\end{itemize}

\begin{definition}
	\label{defa1Ta2T}
	For a PPDO $\tilde{A}$ of symbol $\sigma_{\tilde{A}}$, we define $a_1(\tilde{A})$ and $a_2(\tilde{A})$ by 
	\[a_1(\tilde{A})(x,n) = \frac{\tilde{\sigma}(\arccos(x),n) + \tilde{\sigma}(\arccos(x),-n)}{2}\,,\]
	\[a_2(\tilde{A})(x,n) = \frac{\tilde{\sigma}(\arccos(x),n) - \tilde{\sigma}(\arccos(x),-n)}{2i\sqrt{1-x^2}}\,,\]
	where
	\[\tilde{\sigma}(\theta,n) = \frac{\sigma_{\tilde{A}}(\theta,n) + \sigma_{\tilde{A}}(-\theta,-n)}{2}\,.\]
\end{definition}
\noindent We can now state the main results of this section. All properties of the new class $Op(S^\infty_T)$ follow easily from this theorem, as shown in \autoref{PDOTs} below. Recall the definition of the operator $\mathcal{C}$ from eq.~\eqref{defCS}.  
\begin{theorem}
	\label{CharacSalpha}
	Let $A : T^{\infty} \to T^{-\infty}$. Assume that for some PPDO $\tilde{A} \in Op(\Sigma^\alpha)$, there holds
	\begin{equation}
	\label{eqEquiv}
	\mathcal{C}A = \tilde{A}\mathcal{C} \quad \textup{in  } T^{\infty}\,.
	\end{equation}
	Then $A$ has a unique continuous extension as an element of $Op(S^\alpha_T)$, and $(a_1(\tilde{A}),a_2(\tilde{A}))$ is a pair of symbols for $A$. \\
	Reciprocally let $A = Op_T(a_1,a_2) \in Op(S^\alpha_T)$. Then \eqref{eqEquiv} holds, taking for $\tilde{A}$ the PPDO of order $\alpha$ given by the symbol
	\[\sigma_{\tilde{A}} = \tilde{\sigma}(a_1,a_2)\,.\]
\end{theorem}
\begin{remark}
	We have already stated that a pseudo-differential operator on $T^s$ always admits several distinct symbols. \autoref{CharacSalpha} gives another way to view this fact, by observing that, when $A \in \textit{Op}(S^\alpha_T)$, there is an infinite number of operators $\tilde{A}$ satisfying $\mathcal{C}A = \tilde{A} \mathcal{C}$. Indeed, if $\tilde{B}$ is any PPDO that vanishes on the set of even functions, one has 
	\[\mathcal{C}A =( \tilde{A} + \tilde{B})\mathcal{C} \quad \textup{in } T^{\infty}\,.\]
	In light of this, a natural idea to define uniquely the symbol of $A$ would be to set
	\[a_1 = a_1(\tilde{A}^*), \quad a_2 = a_2(\tilde{A}^*)\,,\] 
	where $\tilde{A}^*$ is the operator defined by
	\begin{align*}
	\tilde{A}^*\mathcal{C}u &= \mathcal{C} Au\,,\\
	\tilde{A}^*\mathcal{S}u &= 0\,.
	\end{align*}
	However, though $A \in Op(S^\alpha_T)$, $\tilde{A}^*$ may fail to be a PPDO of order $\alpha$. To see why, one can check that if $\tilde{A}$ is a PPDO of order $\alpha$ such that $\mathcal{C}A = \tilde{A} \mathcal{C}$, then the symbol of $\tilde{A}^*$ must be given by
	\[\sigma_{\tilde{A}^*}(\theta,n) = \sigma_{\tilde{A}}(\theta,n) + e^{-2in\theta}\sigma_{\tilde{A}}(\theta,-n)\,.\] 
	In general, this symbol is not in $\Sigma^\alpha$ because of the oscillatory term $e^{-2in\theta}$. 
	In conclusion, there is no clear way how to fix a natural representative in the class of pairs $(a_1,a_2)$ that define the same operator $A$. However, although unusual, this is not an obstacle for the theory. 
\end{remark}
\begin{proof}
	\noindent The proof is decomposed into several lemmas, relying mostly on simple algebraic manipulations. Besides that, the key ingredient is that the operators $\mathcal{C}$ and $\mathcal{S}$ are bijective on the sets of smooth even and odd functions respectively. 
	\begin{lemma}
		\label{identiteAlg}
		Let $A = Op_T(a_1,a_2)$ and $\tilde{A} = Op(\tilde{\sigma}(a_1,a_2))$. Then for all $n \in \N$, 
		\[\mathcal{C}AT_n = \tilde{A}\mathcal{C}T_n \,.\]
	\end{lemma}
	\begin{proof}
		Let $A = Op_T(a_1,a_2)$. Let $\sigma =\tilde{\sigma}(a_1,a_2)$, and $\tilde{A} = Op(\sigma)$. Fix $n \in \N$. On the one hand, we can write
		\begin{align*}
		AT_n(\cos(\theta)) &= a_1(\cos \theta,n) T_n(\cos(\theta)) - \omega^2(\cos \theta) a_2(\cos \theta,n) U_{n-1}(\cos \theta)\\
		&= a_1(\cos \theta,n) \cos(n \theta) - \sin \theta a_2(\cos \theta,n) \sin(n\theta)\,.
		\end{align*}
		On the other hand, 
		\begin{align*}
		\tilde{A}(\mathcal{C}T_n)(\theta) & = \frac{\tilde{A}(e_n)(\theta) + \tilde{A}(e_{-n})(\theta)}{2} \\
		= &\frac{\sigma(\theta,n)e^{in\theta} + \sigma(\theta,-n)e^{-in\theta}}{2}\\
		= &\left(\frac{\sigma(\theta,n) + \sigma(\theta,-n) }{2}\right) \cos(n\theta) + i\left(\frac{\sigma(\theta,n) - \sigma(\theta,-n) }{2}\right)\sin(n\theta)\,.
		\end{align*}
		From the definition of $\tilde{\sigma}(a_1,a_2)$, we see that the first term in parenthesis is equal to $a_1(\cos\theta,n)$, and the second is $i \sin \theta a_2(\cos\theta,n)$. Therefore, 
		\[\mathcal{C}AT_n = \tilde{A}\mathcal{C}T_n\]
		for all $n \in \N$ and the result is proved. 
	\end{proof}
	\begin{lemma}
		\label{TransportPPDO_T}
		If $\tilde{A} = Op(\sigma_{\tilde{A}})$ is a PPDO, then the functions $a_1(\tilde{A})$ and $a_2(\tilde{A})$ are $C^\infty$ in the variable $x$ and there holds the identity
		\[\tilde{\sigma}(a_1(\tilde{A}),a_2(\tilde{A}))(\theta,n) = \frac{\sigma_{\tilde{A}}(\theta,n) + \sigma_{\tilde{A}}(-\theta,-n)}{2}\,.\]
	\end{lemma}
	\begin{proof}
		Let $\displaystyle\tilde{\sigma}(\theta,n) = \frac{\sigma_{\tilde{A}}(\theta,n) + \sigma_{\tilde{A}}(-\theta,-n)}{2}$. We decompose $\tilde{\sigma}$ as \[\tilde{\sigma}(\theta,n) = f(\theta,n) + g(\theta,n)\] 
		where $f(\theta,n) = \frac{\tilde{\sigma}(\theta,n) + \tilde{\sigma}(\theta,-n)}{2}$ and $g(\theta,n) = \frac{\tilde{\sigma}(\theta,n) - \tilde{\sigma}(\theta,-n)}{2}$. By construction, $f$ (resp. $g$) is even (resp. odd) in both $\theta$ and $n$.
		For $n > 0$, there holds
		\[a_1(\tilde{A})(x,n) = f(\arccos(x),n),\quad a_2(\tilde{A})(x,n) = \frac{g(\arccos(x),n)}{i\sqrt{1-x^2}}\,.\]
		Recalling \autoref{lemChar}, this is equivalently expressed as
		\[a_1(\tilde{A})(\cdot,n) = \mathcal{C}^{-1}f(\cdot,n), \quad  a_2(\tilde{A})(\cdot,n) = -i\mathcal{S}^{-1}g(\cdot,n)\,.\] 
		Therefore, $a_1(\tilde{A})$ and $a_2(\tilde{A})$ are $C^\infty$ in $x$ since $f$ (resp. $g$) is a smooth even (resp. odd) function. 
		By definition, we have 
		\begin{align*}
		\tilde{\sigma_T}(a_1(\tilde{A}),a_2(\tilde{A}))(\theta,n) &= \mathcal{C}a_1(\tilde{A})(\theta,\abs{n}) + i \text{sign}(n) \mathcal{S}a_2(\tilde{A})(\theta,\abs{n}) \\
		& = f(\theta,\abs{n}) + \text{sign}(n) g(\theta,\abs{n})\\
		& = f(\theta,{n}) + g(\theta,{n})\\
		& = \tilde{\sigma}(\theta,n)\,,
		\end{align*}
		recalling that $f$ (resp. $g$) is even (resp. odd) in $n$. 
	\end{proof}
	\begin{lemma}
		\label{SymetrieSymbol}
		If $\tilde{A} = Op(\sigma_{\tilde{A}})$ is a PPDO that stabilizes the set of smooth even functions, then 
		$\tilde{A}$ coincides on this set with $\tilde{B} = Op(\sigma_{\tilde{B}})$ where 
		\[{\sigma}_{\tilde B}(\theta,n) = \frac{\sigma_{\tilde{A}}(\theta,n) +\sigma_{\tilde{A}}(-\theta,-n)}{2}\,.\]
	\end{lemma}
	\begin{proof}
		Let $u$ be a smooth even function. Since $Au$ is even, we have
		\[Au(\theta) = \frac{Au(\theta) + Au(-\theta)}{2}\,.\]
		Thus
		\begin{align*}
		Au(\theta) &= \frac{1}{2}\left(\sum_{n  \in \Z}\sigma_{\tilde{A}}(\theta,n) \mathcal{F}u(n)e^{in\theta} + \sigma_{\tilde{A}}(-\theta,n) \mathcal{F}u(n)e^{-in\theta}\right)\\
		& = \frac{1}{2}\left(\sum_{n  \in \Z}\sigma_{\tilde{A}}(\theta,n) \mathcal{F}u(n)e^{in\theta} + \sigma_{\tilde{A}}(-\theta,-n) \mathcal{F}u(-n)e^{in\theta}\right)\\
		& = \sum_{n  \in \Z}\frac{\sigma_{\tilde{A}}(\theta,n) + \sigma_{\tilde{A}}(-\theta,-n)}{2} \mathcal{F}u(n)e^{in\theta}\,
		\end{align*}
		since $\mathcal{F}u(n) = \mathcal{F}u(-n)$. This proves the claim. 
	\end{proof}
	\begin{lemma}
		\label{EquivalenceACCA}
		If $A : T^{\infty} \to T^{-\infty}$ is such that there exists a PPDO $\tilde{A}$ satisfying 
		\[\forall n \in \N, \quad \mathcal{C}AT_n = \tilde{A}\mathcal{C}T_n\,,\]
		then $A = Op_T(a_1(\tilde{A}),a_2(\tilde{A}))$. 
	\end{lemma}
	\begin{proof}
		Notice that the assumption implies that $\tilde{A}$ stabilizes the set of smooth even functions. If $\sigma_{\tilde{A}}$ is the symbol of $\tilde{A}$ and 
		\[\sigma(\theta,n) = \frac{\sigma_{\tilde{A}}(\theta,n) + \sigma_{\tilde{A}}(-\theta,-n)}{2}\] then, by \autoref{SymetrieSymbol}, we have $\tilde{A}\mathcal{C} = Op(\sigma)\mathcal{C}$. Moreover, by \autoref{TransportPPDO_T}, we know that
		\[\sigma = \tilde{\sigma}_{T}(a_1(\tilde{A}),a_2(\tilde{A}))\] 
		and thus, letting $B = Op_T(a_1(\tilde{A}),a_2(\tilde{A}))$, \autoref{identiteAlg} tells us that 
		\[\forall n \in \N, \quad \mathcal{C} B T_n = Op(\sigma)\mathcal{C}T_n\,.\]
		Summing up, we have
		\[\forall n \in \N, \quad \mathcal{C} A T_n = \tilde{A} \mathcal{C} T_n =  Op(\sigma)\mathcal{C}T_n = \mathcal{C} B T_n\,.\]
		This ensures $A = B$. 
	\end{proof}
	\noindent The proof of \autoref{CharacSalpha} is concluded as follows. Assume that for any $u \in T^{\infty}$, $\mathcal{C}Au = \tilde{A}\mathcal{C}u$ where $\tilde{A}$ is some PPDO of order $\alpha$ with a symbol $\sigma_{\tilde{A}}$. The linear continuous extension of $A$ is uniquely defined for $u \in T^{-\infty}$ by
	\[\forall v \in T^\infty, \quad \duality{Au}{v}_\frac{1}{\omega} = \duality{\mathcal{C}Au}{\mathcal{C}v}_{L^2_{per}} = \duality{\tilde{A} \mathcal{C}u}{\mathcal{C}v}_{L^2_{per}}\,\]
	where the last quantity makes sense since $\mathcal{C}u \in H^{-\infty}_{per}$. By \autoref{EquivalenceACCA}, we have $A = Op_T(a_1(\tilde{A}),a_2(\tilde{A}))$. It remains to show that $\sigma \isdef \tilde{\sigma}(a_1(\tilde{A}),a_2(\tilde{A})) \in \Sigma^\alpha$. By \autoref{TransportPPDO_T}, if $\tilde{A}$ has a symbol $\sigma_{\tilde{A}}$, then 
	\[\sigma(\theta,n) = \frac{\sigma_{\tilde{A}}(\theta,n) + \sigma_{\tilde{A}}(-\theta,-n)}{2}\,,\]
	and $\sigma_{\tilde{A}} \in \Sigma^\alpha$ immediately implies $\sigma \in \Sigma^\alpha$. This proves the first assertion. The second assertion is an immediate consequence of \autoref{identiteAlg}.  
\end{proof}
\subsubsection*{Extension to smooth open curves}
Recall the definition of the pullback $R_\Gamma$ introduced in section \ref{TsUs(Gamma)}.
\begin{definition}
	Let $A : T^{-\infty}(\Gamma) \to T^{-\infty}(\Gamma)$. We say that $A$ is a pseudo-dif{\-}ferential operator (of order $\alpha$) on $T^{-\infty}(\Gamma)$ if $R_\Gamma A R_\Gamma^{-1} \in Op(S^{\infty}_T)$ ($\in Op(S_T^{\alpha})$). The set of pseudo-differential operators of order $\alpha$ on $T^{-\infty}(\Gamma)$ is denoted by $\textit{Op}(S^{\alpha}_T(\Gamma))$. We say that $(a_1,a_2)$ is a pair of symbols of $A$ if it is a pair of symbols of $R_\Gamma AR_\Gamma^{-1}$. Similarly, $A$ is said to be elliptic if $R_\Gamma A R_{\Gamma}^{-1}$ is elliptic. For $A$ and $B$ in $\textit{Op}(S^{\infty}_T(\Gamma))$, we again write 
	$$A = B + T_\alpha \,$$ 
	if $A - B \in \textit{Op}(S^\alpha_T(\Gamma))$.
\end{definition}
The next result lists some properties of the class $Op(S^\infty_T(\Gamma))$ inherited from $Op(\Sigma^{\infty})$. 
\begin{corollary}
	\label{PDOTs}
	There hold the following properties:
	\begin{itemize}
		\item[(i)] If $A \in \textit{Op}(S^\alpha_T(\Gamma))$, then for all $s \in \R$, $A : T^{s}(\Gamma) \to T^{s -\alpha}(\Gamma)$ is continuous. \\
		
		\item[(ii)]$A \in Op(S^\alpha_T(\Gamma)) \text{ and } B \in Op(S^\beta_T(\Gamma)) \implies AB \in Op(S^{\alpha + \beta}_T(\Gamma))$.
		\medskip
		
		\item[(iii)] If $A$ and $B$ admit the pairs of symbols $(a_1,a_2)$ and $(b_1,b_2)$ respectively, then $AB$ admits the pair of symbol $\displaystyle(a_1(\tilde{C}),a_2(\tilde{C}))$ where 
		\[\tilde{C} = Op(\tilde{\sigma}(a_1,a_2))Op(\tilde{\sigma}(b_1,b_2))\,.\]
		\item[(iv)] If $A \in Op(S^\alpha_T(\Gamma))$ and $B \in Op(S^\beta_T(\Gamma))$, then $[A,B] = AB - BA$ is in $Op(S^{\alpha + \beta -1}_T(\Gamma))$.
		\medskip
		\item[(v)] An operator $A \in Op(S_{T}^{\alpha}(\Gamma))$ is elliptic if and only if there exists an elliptic PPDO $\tilde{A}$ of order $\alpha$ such that $\mathcal{C}R_\Gamma AR_\Gamma^{-1} = \tilde{A}\mathcal{C}$ in $T^{\infty}$. 
	\end{itemize}
\end{corollary}
\begin{proof}
	For the sake of conciseness, we only prove the corollary in the case of the class $Op(S^\infty_T)$ (corresponding to the segment $[-1,1]$). The proofs for a general curve do not contain any additional difficulty. 
	
	\noindent {\em (i)} Let $A \in Op(S^\alpha_T)$, and let $\tilde{A}$ be a PPDO of order $\alpha$ such that \eqref{eqEquiv} holds. Let $u \in T^s$ for some $s \in \R$. Applying the isomorphic property of $\mathcal{C}$ (cf. \autoref{lemChar}) and the continuity of $\tilde{A}$ from $H^s$ to $H^{s-\alpha}$, 
	\[\norm{Au}_{T^s} \leq C \norm{\mathcal{C}Au}_{H^s} \leq C \norm{\tilde{A}\mathcal{C}u}_{H^s} \leq C \norm{\mathcal{C}u}_{H^{s- \alpha}} \leq C \norm{u}_{T^{s-\alpha}}\,.\]
	\noindent {\em (ii)} Let $A = Op_T(a_1,a_2) \in S^\alpha_T$, $B = Op_T(b_1,b_2) \in S^\beta_T$ and let $\tilde{A} = Op(\tilde{\sigma}(a_1,a_2))$, $\tilde{B} = Op(\tilde{\sigma}(a_1,a_2))$. We have 
	\[\mathcal{C}AB = \tilde{A} \mathcal{C}B = \tilde{A}\tilde{B}\mathcal{C} \quad \textup{in } T^{\infty}\,.\]
	Applying the properties of the PPDOs, one has $\tilde{A}\tilde{B} \in Op(\Sigma^{\alpha + \beta})$ therefore, by \autoref{CharacSalpha}, $AB \in Op(S_{T}^{\alpha + \beta})$. \\
	\noindent {\em (iii)} Follows immediately from \autoref{CharacSalpha}.\\
	\noindent {\em (iv)} The commutator of $A$ and $B$ satisfies
	\[\mathcal{C}(AB - BA) = \left(\tilde{A}\tilde{B} -\tilde{B} \tilde{A}\right) \mathcal{C} \quad \textup{in } T^{\infty}\] 
	and $\tilde{B}\tilde{A} - \tilde{A} \tilde{B}$ is a PPDO of order $\alpha + \beta - 1$. \\
	\noindent {\em (v)} If $A = Op_T(a_1,a_2)$ is elliptic, then $\tilde{\sigma}(a_1,a_2)$ is elliptic and thus $\tilde{A} = Op(\tilde{\sigma})$ is elliptic. By \autoref{CharacSalpha}, we have $\mathcal{C}A = \tilde{A} \mathcal{C}$ in $T^\infty$. 
	Reciprocally, let $\tilde{A} = Op(\sigma_A)$ be an elliptic PPDO of order $\alpha$ such that $\mathcal{C}A = \tilde{A} \mathcal{C}$ in $T^\infty$. Then, $A$ admits the pair of symbols $a_1(\tilde{A}),a_2(\tilde{A})$ and by \autoref{TransportPPDO_T}, we have
	\[\tilde{\sigma}(a_1(\tilde{A}),a_2(\tilde{A}))(\theta,n) = \frac{\sigma_A(\theta,n) + \sigma_A(-\theta,-n)}{2}\,.\]
	It is easy to check that the last symbol is elliptic when $\sigma_A$ is elliptic. This proves that $A$ is elliptic. 
\end{proof}
\begin{remark}
	\label{RemSymb}
	The item {\em (iii)} above provides a symbolic calculus on the class \linebreak $S^\alpha_T(\Gamma)$ as follows. If $B$ and $C$ respectively admit the pair of symbols $(b_1,b_2)$ and $(c_1,c_2)$, then $BC$ admits the pair of symbols 
	\[(b_1,b_2) \#_T (c_1,c_2) \isdef \left(a_1(\tilde{A}),a_2(\tilde{A})\right)\,\]
	where $\tilde{A} = \textit{Op}\left(\tilde{\sigma}(b_1,b_2) \# \tilde{\sigma}(c_1,c_2) \right)$.
	One can use \eqref{a_diese_b} to compute an asymptotic expansion of $\tilde{\sigma}(b_1,b_2) \# \tilde{\sigma}(c_1,c_2)$ which, in turn, gives an asymptotic expansion of $(b_1,b_2) \#_T (c_1,c_2)$. 
\end{remark}

\begin{lemma}
	\label{lemParametrixOpS}
	Let $A = Op(\sigma_A)$ be an elliptic PPDO whose symbol satisfies
	\begin{equation}
	\label{symmProp}
	\sigma_A(\theta,n) = \sigma_A(-\theta,-n)\,.
	\end{equation}
	Then there exists an elliptic parametrix $B = Op(\sigma_B)$ where $\sigma_B$ satisfies the same symmetry. 
\end{lemma}
\begin{proof}
	Let us fix an elliptic PPDO $A$ of order $\alpha$ and let $\sigma_A$ be a prolongated symbol of $A$ that we may assume to have the property
	\[\sigma_A(\theta,\xi) = \sigma_A(-\theta,-\xi)\,.\]
	Let $B_1 = Op(\sigma_1)$, with the prolongated symbol $\sigma_1 \in \Sigma^{-\alpha}$, be a parametrix for $A$, and let 
	\[\sigma_2(\theta,\xi) = \frac{\sigma_1(\theta,\xi) + \sigma_1(-\theta,-\xi)}{2}\,.\]
	This symbol in in $\Sigma^{-\alpha}$, has the desired symmetry and it remains to show that 
	\[AB_2 = I_d + \Sigma_{-\infty}, \quad B_2A = I_d + \Sigma_{-\infty}\] 
	where $B_2 \isdef Op(\sigma_2)$ .
	We show for example the first equality, the second one being similar. Let $n \in \N$. We have by symbolic calculus (cf. eq.~\eqref{a_diese_b}):
	\[\sigma_C(\theta,n) = \sum_{j = 0}^{n}\frac{1}{j!} \left(\frac{\partial}{\partial_\xi}\right)^{j}\sigma_A(\theta,n) D_\theta^j \sigma_2(\theta,n) + \sigma_R(\theta,n)\]
	with $\sigma_R \in \Sigma^{-n-1}$. Replacing $\sigma_2$ by its expression, this yields
	\[\begin{split}
	\sigma_C(\theta,\xi) = &\frac{1}{2} \sum_{j = 0}^{n} \frac{1}{j!}\left[\left(\frac{\partial}{\partial_\xi}\right)^{j}\sigma_A(\theta,n) D_\theta^j \sigma_1(\theta,n)\right. \\
	&+ (-1)^{j} \left.\left(\frac{\partial}{\partial_\xi}\right)^{j} \sigma_A(\theta,n)D_\theta^j\sigma_1(-\theta,-n)\right] + \sigma_R(\theta,\xi)\,.
	\end{split}\]
	Using the symmetry property of $\sigma_A$, we obtain
	\[\sigma_C(\theta,\xi) = \frac{\sigma_n(\theta,\xi) + \sigma_n(-\theta,-\xi)}{2} +  \sigma_R(\theta,\xi)\]
	where 
	\[\sigma_n(\theta,\xi) = \sum_{j = 0}^{n} \frac{1}{j!} \left(\frac{\partial}{\partial_\xi}\right)^{j}\sigma_A(\theta,n) D_\theta^j \sigma_1(\theta,n)\,.\]
	But by eq.~\eqref{a_diese_b}  we have $\sigma_n = \sigma_A \# \sigma_1 + \Sigma^{-n-1}$ and since $B_1$ is a parametrix of $A$, $\sigma_A \# \sigma_1 = 1 + \Sigma^{-\infty}\,$. Consequently, there exists a symbol $\sigma_S \in \Sigma^{-n-1}$ such that $\sigma_n = 1 + \sigma_S$.
	Thus
	\[\sigma_C(\theta,\xi) - 1 = \sigma_R(\theta,\xi) + \frac{\sigma_S(\theta,\xi) + \sigma_S(-\theta,-\xi)}{2} \in \Sigma^{-n-1}\,.\]
	Since we have established this for all $n \in \N$, we have proved $AB_2 = I_d + \Sigma_{-\infty}.$ as announced. 
\end{proof}
\begin{corollary}
	\label{parametrixPDOTs}
	Let $A \in Op(S^\alpha_T(\Gamma))$ be elliptic. Then there exists $B \in Op(S^{-\alpha}_T(\Gamma))$ elliptic such that 
	\[BA = I_d + T_{-\infty}, \quad AB = I_d + T_{-\infty}\,.\]
\end{corollary}
\begin{proof}
	Here again, we treat only the particular case of $Op(S^\alpha_T)$ for conciseness.
	Let $A \in Op(S^\alpha_T)$ be elliptic. By \autoref{PDOTs}, there exists an elliptic PPDO $\tilde{A} \in Op(\Sigma^\alpha)$ such that 
	\[\mathcal{C}A = \tilde{A} \mathcal{C} \quad \textup{in } T^{\infty}\,.\]
	Such a PPDO necessarily preserves the set of smooth even functions. Thus, by \autoref{SymetrieSymbol}, it may be assumed that its symbol $\sigma_{\tilde{A}}$ has the symmetry property \eqref{symmProp}. By the previous lemma, let $\tilde{B} \in Op(\Sigma^{-\alpha})$ be a parametrix of $\tilde{A}$ whose symbol $\sigma_{\tilde{B}}$ possesses this symmetry, and let 
	\[B = Op(a_1(\tilde{B}),a_2(\tilde{B}))\,.\]
	By \autoref{TransportPPDO_T}, we have
	\[\tilde{\sigma}(a_1(\tilde{B}),a_2(\tilde{B})) = \sigma_{\tilde{B}}\]
	hence $B \in Op(S^{-\alpha}_T)$. Moreover, there holds
	$\mathcal{C}B = \tilde{B} \mathcal{C}$
	by \autoref{CharacSalpha}. Finally, we have
	\[\mathcal{C} (AB-I_d) = (\tilde{A} \tilde{B} - I_d) \mathcal{C} = (I_d + R)\mathcal{C} \quad \textup{in } T^{\infty}\]
	where $R$ is a smoothing PPDO. This proves that $AB - I_d \in Op(S^{-\infty}_T)$ and the same arguments show that $BA - I_d$ also belongs to $Op(S_T^{-\infty})$. 
\end{proof}

\subsection{Pseudo-differential operators on \texorpdfstring{$U^s(\Gamma)$}{Us(Gamma)}}

\label{subsec:PDOUs}

We proceed to introduce an analogous family of pseudo-differential operators defined this time on the spaces $U^s$. Similar properties hold for this new family of pseudo-differential operators. They are stated here but the proofs do not differ in any significant way from the previous, and are thus omitted.

\begin{definition}
	\label{DefOpUs}
	Let $A$ be an operator on $U^{-\infty}$ and assume that there exists a couple of smooth functions $a_1$ and $a_2$ defined on $[-1,1]\times \N$, that are $\Cinf$ in the first argument and such that for all $n \in \N^*$,
	\begin{equation}
	AU_{n-1} = a_1(x,n) U_{n-1} + a_2(x,n) T_{n}\,.
	\label{defOpUs}
	\end{equation}
	The operator defined by the previous formula is denoted by $\textit{Op}_U(a_1,a_2)$. For $n\in \Z$ and $\theta \in [0,2\pi]$, define the symbol $\tilde{\sigma}(a_1,a_2)$ as before by
	\[\tilde{\sigma}(a_1,a_2)(\theta,n) =  a_1(\cos\theta,\abs{n}) + i\sin\theta\,\textup{sign}(n) a_2(\cos\theta,\abs{n})\,\]
	with the convention $a_1(x,0) = a_2(x,0) = 0$. We say that $(a_1,a_2) \in S^\alpha_U$ if \linebreak $\tilde{\sigma}(a_1,a_2) \in \Sigma^\alpha$.  In this case, we say that $A$ is a pseudo-differential operator on $U^s$ and the (non-unique) couple of functions $(a_1,a_2)$ is called a pair of symbols of $A$. We also take the notation $S^{\infty}_U \isdef \cup_{\alpha \in \R} S^\alpha_U$ and $S^{-\infty}_U = \cap_{\alpha \in \R} S^\alpha_U$.   and the set of pseudo-differential operators (of order $\alpha$) in $U^{-\infty}$ by $\textit{Op}(S^\infty_U)$ (by $\textit{Op}(S^\alpha_U)$). The operator $A$ is said to be elliptic if it admits a pair of symbols $(a_1,a_2)$ such that $\tilde{\sigma}(a_1,a_2)$ is elliptic. Finally, if $A, B \in Op(S^{\infty}_U)$ are such that $A - B \in Op(S^{\alpha}_U)$, we write
	\[A = B + U_{\alpha}\,.\]
\end{definition} 
\noindent Recall the definition of the isometric mapping $\mathcal{S}$ from \eqref{defCS}. 
\begin{theorem}
	\label{CharacSalphaU}
	Let $A : U^{\infty} \to U^{-\infty}$. Assume that for some PPDO $\tilde{A} \in Op(\Sigma^\alpha)$, there holds
	\begin{equation}
	\label{eqEquivU}
	\mathcal{S}A = \tilde{A}\mathcal{S}\, \quad \textup{in } U^{\infty}\,.
	\end{equation}
	Then $A$ has a unique continuous extension as an element of $Op(S^\alpha_U)$, and $(a_1(\tilde{A}),a_2(\tilde{A}))$ is a pair of symbols for $A$. \\
	Reciprocally let $A = Op_U(a_1,a_2) \in Op(S^\alpha_U)$. Then \eqref{eqEquivU} holds, taking for $\tilde{A}$ the PPDO of order $\alpha$ given by the symbol
	\[\sigma_{\tilde{A}} = \tilde{\sigma}(a_1,a_2)\,.\]
\end{theorem}
We also extend this notion to open curves:
\begin{definition}
	Let $A : U^{-\infty}(\Gamma) \to U^{-\infty}(\Gamma)$. We say that $A$ is a pseudo-differential operator (of order $\alpha$) on $U^{-\infty}(\Gamma)$ if $R_\Gamma A R_\Gamma^{-1}$ belongs to $Op(S^{\infty}_U)$ (to $Op(S_U^{\alpha})$). The set of pseudo-differential operators of order $\alpha$ on $U^{-\infty}(\Gamma)$ is denoted by $\textit{Op}(S^{\alpha}_U(\Gamma))$. We say that $(a_1,a_2)$ is a pair of symbols of $A$ if it is a pair of symbols of $R_\Gamma AR_\Gamma^{-1}$. The operator $A$ is said to be elliptic if $R_\Gamma A R_\Gamma^{-1}$ is elliptic. For $A$ and $B$ in $\textit{Op}(S^{\infty}_U(\Gamma))$, we again write $A = B + U_\alpha$ if $A - B \in \textit{Op}(S^\alpha_U(\Gamma))$.
\end{definition}
\begin{corollary}
	\label{PDOUs}
	There hold the following properties:
	\begin{itemize}
		\item[(i)] If $A \in \textit{Op}(S^\alpha_U(\Gamma))$, then for all $s \in \R$, $A : U^{s}(\Gamma) \to U^{s -\alpha}(\Gamma)$ is continuous. \\
		
		\item[(ii)]$A \in Op(S^\alpha_U(\Gamma)) \text{ and } B \in Op(S^\beta_U(\Gamma)) \implies AB \in Op(S^{\alpha + \beta}_U(\Gamma))$.
		\medskip
		
		\item[(iii)] If $A$ and $B$ admit the pairs of symbols $(a_1,a_2)$ and $(b_1,b_2)$ respectively, then $AB$ admits the pair of symbol $\displaystyle(a_1(\tilde{C}),a_2(\tilde{C}))$ where 
		\[\tilde{C} = Op(\tilde{\sigma}(a_1,a_2))Op(\tilde{\sigma}(b_1,b_2))\,.\]
		\item[(iv)] If $A \in Op(S^\alpha_U(\Gamma))$ and $B \in Op(S^\beta_U(\Gamma))$, then $[A,B] = AB - BA$ is in $Op(S^{\alpha + \beta -1}_U(\Gamma))$.
		\medskip
		\item[(v)]  An operator $A \in Op(S_{U}^{\alpha}(\Gamma))$ is elliptic if and only if there exists an elliptic PPDO $\tilde{A}$ of order $\alpha$ such that 
		\[\mathcal{S} R_\Gamma AR_\Gamma^{-1} = \tilde{A} \mathcal{S} \quad \textup{in } U^{\infty}\,.\] 
		In this case, there exists an elliptic operator $B\in Op(S_{U}^{-\alpha}(\Gamma))$ such that \[BA = I_d + U_{-\infty}, \quad AB = I_d + U_{-\infty}\,.\]
	\end{itemize}
\end{corollary}

\subsection{Connections between the two classes}
\label{subsec:linkTsUs}
\begin{lemma}
	\label{LemdxAomegadeomega}
	Let $A \in \textit{Op}(S_T^{\alpha}(\Gamma))$ and 
	$$B = -\partial_\tau A \omega_\Gamma \partial_\tau \omega_\Gamma\,.$$ 
	Then $B \in \textit{Op}(S_U^{\alpha+2}(\Gamma))$ and if $\tilde{A}$ is a PPDO such that $\mathcal{C}R_\Gamma AR_\Gamma^{-1} = \tilde{A} \mathcal{C}$, then 
	$$\mathcal{S}R_\Gamma BR_\Gamma ^{-1} = -\partial_\theta \tilde{A} \partial_\theta \mathcal{S} \quad \textup{in } U^{\infty}\,.$$
\end{lemma}
\begin{proof}
	One can check the following identities:
	\begin{align*}
	\partial_\theta \mathcal{S} &= - \mathcal{C} \omega \partial_x \omega \,,\\
	\partial_\theta \mathcal{C} &= - \mathcal{S} \partial_x\,.
	\end{align*}
	Let $A' = R_\Gamma AR_\Gamma^{-1}$ and $B' = R_\Gamma BR_\Gamma^{-1}$. Assuming that $\mathcal{C} A' = \tilde{A} \mathcal{C}$, there holds
	\begin{align*}
	\mathcal{S}B' &= -\mathcal{S}R \partial_\Gamma A \omega_\Gamma \partial_\Gamma \omega_\Gamma R^{-1}\\
	&= -\mathcal{S}\partial_x A' \omega \partial_x \omega\\
	&=\partial_\theta \mathcal{C}A'\omega \partial_x \omega\\
	&=\partial_\theta \tilde{A} \mathcal{C} \omega \partial_x\omega\\
	&= -\partial_\theta \tilde{A} \partial_\theta \mathcal{S}\,.
	\end{align*}
	Since $\tilde{A}$ can be chosen as a PPDO of order $\alpha$ by \autoref{CharacSalpha}, $\partial_\theta \tilde{A} \partial_\theta$ is then a PPDO of order $\alpha +2 $ from which we conclude that $B \in \textit{Op}(S_U^{\alpha + 2}(\Gamma))$.  
\end{proof}
\begin{lemma}	
	\label{LemAomega2}
	Let $A \in \textit{Op}(S_T^{\alpha}(\Gamma))$ and let
	$$B = A \omega_\Gamma^2\,.$$ 
	Then $B \in \textit{Op}(S_U^{\alpha}(\Gamma))$ and if $\tilde{A}$ is a PPDO such that $\mathcal{C}R_\Gamma AR_\Gamma^{-1} = \tilde{A} \mathcal{C}$, then 
	$$\mathcal{S}R_\Gamma BR_\Gamma^{-1} = \sin\tilde{A} \sin \mathcal{S} \quad \textup{in } U^{\infty}\,,$$ 
	where $\sin$ denotes the operator $f(\theta) \mapsto \sin(\theta) f(\theta)$. 
\end{lemma}
\begin{proof}
	Using the identities
	\begin{equation*}
	\mathcal{S} = \sin \mathcal{C}, \quad \mathcal{C}\omega^2 = \sin \mathcal{S}\,,
	\end{equation*}
	valid in $T^{\infty} = U^{\infty}$, the result follows with a similar proof as above. 
\end{proof}

\subsection{Square-root of pseudo-differential operators}
\label{subsec:squareRoot}

For a self adjoint operator $A$ on some Hilbert space $H$ it is possible to define the (principal) square root $\sqrt{A}$ of $A$ by functional calculus. By principal square root we mean
\[\forall r > 0\,, \quad \sqrt{-r} \isdef i \sqrt{r}\,.\]
For details, we refer the reader to e.g. \cite[Def. 10.5]{hall2013quantum}. It turns out that the square root of a self-adjoint elliptic PPDO is again a PPDO. 
\begin{proposition}
	Let $A$ be a PPDO of order $\alpha > 0$. Assume that $A$ is self-adjoint, elliptic, classical and invertible. Then the operator $\sqrt{A}$ is a self-adjoint, elliptic, classical and invertible PPDO of order $\frac{\alpha}{2}$. If the principal symbol of $A$ is $\sigma_\alpha(\theta,n)$, then the principal symbol of $\sqrt{A}$ is given by $\sqrt{\sigma_\alpha(\theta,n)}$ for $|n|$ sufficiently large. 
\end{proposition}
This result is classical \cite{seeley1967complex,strichartz1972functional}, \cite[Chap. 12]{taylorpseudodifferential} (in those works, the authors study the operators $A^s$ for any $s \in \mathbb{C}$). Note that those proofs take place on the setting of classical pseudo-differential operators on a manifold, but McLean showed in \cite{mclean1991local} that the definition of PPDOs is equivalent to that of usual pseudo-differential operators on the torus. The assumption of invertibility is not essential in the  case of the square root ($s = \frac{1}{2}$), as shown in the next lemma. 
\begin{lemma}
	\label{lemCalculDesSqr}
	Let $A$ be a classical elliptic self-adjoint PPDO of order $\alpha > 0$. There exists a classical elliptic positive definite PPDO of order $\alpha$, $B$, such that 
	\[A = B + R \quad \textup{and} \quad \sqrt{A} = \sqrt{B} + \sqrt{R}\]
	where $R \in Op(\Sigma^{-\infty})$ and $\sqrt{R} \in Op(\Sigma^{-\infty})$. As a consequence, $\sqrt{A}$ is a classical elliptic PPDO of order $\frac{\alpha}{2}$. 
\end{lemma}
\begin{proof}
	By \autoref{SpectralTheoremPPDO}, let us write 
	\[\forall \phi \in C^{\infty}(\mathbb{T}_{2\pi}), \quad A\phi = \sum_{i = 0}^{+ \infty} \lambda_i \duality{\phi}{u_i}_{L^2_{per}}u_i\]
	where $u_i \in C^{\infty}(\mathbb{T}_{2\pi})$. Let $I$ be such that
	\[\lambda_I \leq 0, \quad \lambda_{I + 1}> 0\]
	and let
	\[Ru = \sum_{i \leq I} \lambda_i\duality{u}{u_i}_{L^2_{per}} u_i, \quad Bu = \sum_{i > I} \lambda_i\duality{u}{u_i}_{L^2_{per}} u_i\,.\]
	We indeed have $A = B + R$ and $\sqrt{A} = \sqrt{B} + \sqrt{R}$. It remains to show that $R$ and $\sqrt{R}$ are smoothing operators. We can write $R$ under the form
	\[R\phi(\theta) = \frac{1}{2\pi}\int_{- \pi}^{\pi}\left(\sum_{i = 0}^{I} \lambda_i u_i(\theta) u_i(\theta')\right)\phi(\theta')d\theta'\]
	which, by \autoref{thrunen}, is indeed a smoothing operator. The same reasoning can be applied to $\sqrt{R}$, hence the result is proved. 
\end{proof}
\begin{lemma}
	Let 
	\[\tilde{D}\phi(\theta) = -\partial_{\theta\theta}\phi(\theta) - k^2 \sin^2\theta \phi(\theta)\,.\]
	Then $\sqrt{\tilde{D}}$ is a self-adjoint, classical and elliptic PPDO of order $1$, with a principal symbol given by
	\[\sigma_{1}(\theta,n) = n\,.\]
	Furthermore, there holds
	\[\mathcal{C}R_\Gamma\sqrt{-(\omega_\Gamma \partial_\tau)^2 - k^2 \omega^2_\Gamma} = \sqrt{\tilde{D}} \mathcal{C}R_\Gamma\, \quad \textup{in } T^{-\infty}(\Gamma)\]
	and
	\[\mathcal{S}R_\Gamma\sqrt{-(\partial_\tau \omega_\Gamma)^2 - k^2 \omega^2_\Gamma} = \sqrt{\tilde{D}} \mathcal{S}R_\Gamma\, \quad \textup{in } U^{-\infty}(\Gamma)\]
\end{lemma}
\begin{proof}
	Let $D_1 = -(\omega_\Gamma \partial_\tau)^2 - k^2 \omega_\Gamma^2$ and $D_2 =  -(\partial_\tau \omega_\Gamma )^2 - k^2 \omega_\Gamma^2$. Computing with smooth functions and using density arguments, one has
	\begin{equation}
	\label{rel1sqr}
	\mathcal{C} {R_\Gamma} D_1  = \tilde{D} \mathcal{C} R_\Gamma \quad \textup{in } T^{-\infty}(\Gamma) \,,
	\end{equation}
	\begin{equation}
	\label{rel2sqr}
	\mathcal{S} {R_\Gamma} D_2  = \tilde{D} \mathcal{S}R_\Gamma \quad \textup{in } U^{-\infty}(\Gamma)\,.
	\end{equation}
	As a consequence, since
	\[\mathcal{C} : T^{-\infty} \to H^{-\infty}_e, \quad \mathcal{S} : U^{-\infty} \to H^{-\infty}_o\]
	are isomorphisms, $\tilde{D}$ stabilizes the set of even and odd functions $H^{-\infty}_{e}$ and $H^{-\infty}_{o}$. Since $\tilde{D}$ is a self-adjoint, compact resolvent operator on $H^1_{per}$, its restrictions to $H^{1}_e$ and $H^1_o$ are also self-adjoint and compact resolvent. Therefore, one can find two sets of eigenfunctions $(\varphi_i)_i$ and $(\psi_i)_i$ of $\tilde{D}$, with associated real eigenvalues $(\lambda_i)_i$ and $(\mu_i)_i$, such that $(\varphi_i)$ and $(\psi_i)$ are Hilbert basis of $L^2_e$ and $L^2_o$ respectively. The families $(e_i)_i$ and $(f_i)_i$  defined by 
	\[\mathcal{C} R_\Gamma e_i = \varphi_i, \quad  \mathcal{S}R_\Gamma f_i = \psi_i\]
	provide Hilbert basis of $T^0(\Gamma)$ and $U^{0}(\Gamma)$, and satisfy
	\[D_1 e_i = \lambda_i e_i, \quad D_2 f_i = \mu_i f_i\]
	by eqs.~\eqref{rel1sqr} and \eqref{rel2sqr}. By definition, 
	\[\sqrt{D_1}e_i = \sqrt{\lambda_i} e_i, \quad \sqrt{D_2}f_i = \sqrt{\mu_i} f_i\,,\]
	while
	\[\sqrt{\tilde{D}}\varphi_i =  \sqrt{\lambda_i} \varphi_i, \quad \sqrt{\tilde{D}}\psi_i =  \sqrt{\mu_i} \psi_i.\]
	As a consequence,
	\begin{equation}
	\forall i, \quad \mathcal{C} {R_\Gamma} \sqrt{D_1} e_i   = \sqrt{\tilde{D}} \mathcal{C} R_\Gamma e_i \,,
	\end{equation}
	\begin{equation}
	\forall i, \quad \mathcal{S} {R_\Gamma} \sqrt{D_2} f_i  = \sqrt{\tilde{D}} \mathcal{S}R_\Gamma f_i\,.
	\end{equation}	
	Since $(e_i)$ and $(f_i)$ are Hilbert bases of $T^{0}(\Gamma)$ and $U^0(\Gamma)$, this implies the result by density of those spaces in $T^{-\infty}(\Gamma)$ and $U^{-\infty}(\Gamma)$. 
\end{proof}

\section{Weighted layer potentials on open curves}

We now introduce the weighted single and hypersingular layer potentials on open curves for the Laplace and Helmholtz equations. In the Laplace case ($k = 0$), the layer potentials possess some explicit properties which allow us to relate the spaces $T^s$ to standard Sobolev spaces when $s = \pm \frac{1}{2}$. This is useful to analyze the mapping properties when $k \neq 0$.

\subsection{First-kind integral equations}

Recall the definition and parametrization of the curve $\Gamma$ detailed in section \ref{TsUs(Gamma)}. The single-layer and hypersingular operators, $S_k$ and $N_k$, are defined for all $x \in \Gamma$ by
\begin{equation}
\begin{split}
\quad (S_k \lambda)(x) &= \int_{\Gamma} G_k(x-y) \lambda(y) d\sigma_y\,,\\ 
(N_k \mu) (x) &= -\lim_{\varepsilon \to 0^+} \frac{\partial}{ \partial \varepsilon} \int_{\Gamma} N(y) \cdot \nabla G_k(x + \varepsilon N(x) - y) \mu(y) d\sigma_y
\end{split}
\label{defNk}
\end{equation}
for $x \in \Gamma$, with the Green function $G_k$  defined by
\begin{equation}
\left\{
\begin{aligned}
G_0(z) &= -\dfrac{1}{2\pi} \ln \abs{z}, && \text{ if } k= 0,\\
G_k(z) &= \frac{i}{4}H_0(k|z|), && \text{ if } k > 0,
\end{aligned} 
\right.
\end{equation} 
where $H_0$ is the Hankel function of the first kind. It is known that $S_k$ maps bijectively $\tilde{H}^{-1/2}(\Gamma)$ to $H^{1/2}(\Gamma)$ except for $k = 0$ where $S_k$ has a non-trivial kernel if and only if the logarithmic capacity of $\Gamma$ is $1$, \cite[Theorem 1.8]{wendland1990hypersingular}. On the other hand, $N_k$ maps bijectively $\tilde{H}^{1/2}(\Gamma)$ to $H^{-1/2}(\Gamma)$ \cite[Theorem 1.4]{wendland1990hypersingular}. Here the Sobolev spaces are defined as in \cite[Chap. 3]{mclean2000strongly} with the same notation.

The kernel of the hypersingular operator has a non-integrable singularity, but computations are facilitated by the following formula, valid for smooth functions 
$\mu$ and $\nu$ that vanish at the extremities of $\Gamma$: 
\begin{eqnarray}
\label{NkenfonctiondeSk}
\duality{N_k \mu}{\nu} &=& \int_{\Gamma\times \Gamma} G_k(x-y) \partial_\tau \mu(x) \partial_{\tau}\nu(y) \nonumber\\
&& \quad - k^2 G_k(x,y) \mu(x) \nu(y) n(x) \cdot n(y) d\sigma_x d\sigma_y\,.
\end{eqnarray}
For the geometry under consideration, the solutions $\lambda$ and $\mu$ of the equations 
\[S_k \lambda = u_D, \quad N_k \mu = u_N\]
have singularities (even for $\Cinf$ data $u_D$ and $u_N$) due to the edges of the scatterer (see e.g. \cite[Cor A.5.1]{costabel2003asymptotics}). This encourages to introduce weighted versions of the usual layer potentials as in \cite{bruno2012second}, known to enjoy better mapping properties than $S_k$ and $N_k$. Namely, we define
\begin{equation}
\label{weightedBIOs}
S_{k,\omega_\Gamma} \phi \isdef S_k \left(\frac{\phi}{\omega_\Gamma}\right)\,,\quad N_{k,\omega_\Gamma}\phi \isdef N_k (\omega_\Gamma\phi)\,,
\end{equation}
and recast those equations as 
\begin{equation}
\label{weightedBIEs}
S_{k,\omega_\Gamma}\alpha = u_D\,, \quad N_{k,\omega_\Gamma}\beta = u_N\,
\end{equation}
where the unknowns $\alpha$ and $\beta$ are related to $\lambda$ and $\mu$ by 
\[\lambda = \frac{\alpha}{\omega_\Gamma}, \quad \mu = \omega_\Gamma \beta\,.\]
From eq.~\eqref{NkenfonctiondeSk}, we obtain the following relation between $N_{k,\omega_\Gamma}$ and $S_{k,\omega_\Gamma}$.
\begin{lemma}
	\label{NkomegaSkomega}
	There holds 
	\[N_{k,\omega_\Gamma} = -\partial_\tau S_{k,\omega_\Gamma} \omega_\Gamma \partial_\tau \omega_\Gamma - k^2 V_k \omega_\Gamma^2\]
	where $V_k$ is the integral operator defined by 
	\[V_k u = \int_{\Gamma} \frac{G_k(x - y) N(x) \cdot N(y) u(y)}{\omega_\Gamma(y)} d\sigma_y\,. \]
\end{lemma}
\begin{proof}
	Eq. \eqref{NkenfonctiondeSk} can be rewritten equivalently as 
	\[N_k u = -\partial_\tau S_k \partial_\tau u - k^2 \int_{\Gamma} G_k(x - y) N(x) \cdot N(y) u(y) d\sigma_y\,. \]
	Using the definitions of $N_{k,\omega_\Gamma}$ and $S_{k,\omega_\Gamma}$, the results follow from simple manipulations on this expression. 
\end{proof}

\subsection{Laplace weighted layer potentials on the flat segment}
\label{subsec:weightedLayerPotFlatSeg_II}

We restrict our attention to the case where the wavenumber $k$ is equal to $0$ and $\Gamma = [-1,1]\times{\{0\}}$. The parametrization $r$ is then the constant function equal to $1$, $\partial_\tau = \partial_x$ and $\omega_\Gamma = \omega$. In this context, the weighted potentials are thus denoted by $S_{0,\omega}$ and $N_{0,\omega}$. The following well-known result plays a fundamental role in this work. 
\begin{lemma}
	The weighted layer potentials satisfy
	\begin{equation}
	\label{explicitEigs}
	S_{0,\omega} T_n = \sigma_n T_n
	\end{equation}
	where 
	\begin{equation}
	\label{DefSigman}
	\sigma_n = \begin{cases}
	\dfrac{\ln(2)}{2} & \text{if } n=0,\\
	& \\
	\dfrac{1}{2n} & \text{otherwise},
	\end{cases}
	\end{equation}
	and 
	\[N_{0,\omega} U_n = \frac{(n+1)}{2} U_n\,.\]
\end{lemma}
\begin{proof}
	For the first identity , we refer the reader to e.g. \cite[Theorem 9.2]{mason2002chebyshev}. We wish to show how the second property is deduced from the first. For this, we use eq.~\eqref{NkomegaSkomega} which in this context takes the form
	\[\duality{N_{0,\omega} \beta}{\beta'}_\omega = \duality{S_{0,\omega}  (\omega \partial_x \omega) \beta}{ (\omega \partial_x \omega) \beta'}_\frac{1}{\omega}\,,\]
	that is, 
	\[N_{0,\omega} = -\partial_x S_{0,\omega} \omega \partial_x \omega\,.\] 
	The result follows from
	$\omega \partial_x\omega U_n = (n+1)T_{n+1}$ and ${\partial_x T_{n+1} = -(n+1) U_n}$. 
\end{proof}
\noindent The operators $S_{0,\omega}$ and $N_{0,\omega}$ are thus pseudo-differential operators in $T^s$ and $U^s$ respectively. 
\subsection{Characterization of $T^{\pm 1/2}$ and $U^{\pm 1/2}$}
The next result, and \autoref{LemU12} stated below are equivalent to results formulated in \cite{jerez2012explicit}, equations (4.77-4.86), and Propositions 3.1 and 3.3. A proof is included here for the reader's convenience. 
\begin{lemma}
	\label{LemmaT-1/2}
	We have $T^{-1/2} = \omega\tilde{H}^{-1/2}(-1,1)$ with
	\[\forall u \in \tilde{H}^{-\frac{1}{2}}(-1,1)\,, \quad\norm{u}_{\tilde{H}^{-1/2}} \sim \norm{\omega u}_{T^{-1/2}}\,.\] 
	On the other hand, $T^{1/2} = H^{1/2}(-1,1)$ with 
	\[\forall u \in {H}^{\frac{1}{2}}(-1,1)\,, \quad \norm{u}_{H^{1/2}} \sim \norm{u}_{T^{1/2}}\,.\]
	Here, the symbol $\sim$ denotes the equivalence of the norms.
\end{lemma}
\begin{proof}
	Since the logarithmic capacity of the segment is $\frac{1}{4}$, the (unweighted) single-layer operator $S_0$ is positive and bounded from below on $\tilde{H}^{-1/2}(-1,1)$, (see \cite{mclean2000strongly} chap. 8). Therefore the norm on $\tilde{H}^{-1/2}(-1,1)$ is satisfies to 
	\[\norm{u}_{\tilde{H}^{-1/2}} \sim \sqrt{\duality{S_0u}{u}},\]
	where $\sim$ denote
	On the other hand, the explicit expression \eqref{explicitEigs} implies that if $\alpha\in T^{-1/2}$, then
	\[ \norm{\alpha}_{T^{-1/2}} \sim \sqrt{\duality{S_{0,\omega} \alpha}{\alpha}_\frac{1}{\omega}}.\]
	It remains to notice that, since $\alpha=\omega u$, $\duality{S_{0,\omega} \alpha}{\alpha}_\frac{1}{\omega} = \duality{S_0u}{u}$. This proves the first result. For the second result, we know that, 
	\[(H^{1/2}(-1,1))' =  \tilde{H}^{-1/2}(-1,1)\] 
	(cf. \cite[Chap. 3]{mclean2000strongly} taking the identification with respect to the usual $L^2$ duality denoted by $\duality{\cdot}{\cdot}$) and therefore
	\[\norm{u}_{H^{\frac{1}{2}}} = \sup_{ v\neq 0} \dfrac{\duality{u}{v}}{\norm{v}_{\tilde{H}^{-\frac{1}{2}}}}\,.\]
	According to the previous result, for all $v\in \tilde{H}^{-\frac{1}{2}}$, the function $\alpha = \omega v$ is in $T^{-1/2}$, and $\norm{ v}_{\tilde{H}^{-1/2}} \sim \norm{\alpha}_{T^{-1/2}}$, while $\duality{u}{v} = \duality{u}{\alpha}_\frac{1}{\omega}$. Thus 
	\[\norm{u}_{H^{1/2}} \sim \sup_{\alpha \neq 0} \dfrac{\duality{u}{\alpha}_\frac{1}{\omega}}{\norm{\alpha}_{T^{-1/2}}}\,.\]
	The last quantity is the $T^{1/2}$ norm of $u$ since $T^{1/2}$ is identified to the dual of $T^{-1/2}$ for $\duality{\cdot}{\cdot}_\frac{1}{\omega}$, concluding the proof. 
\end{proof}
\noindent With the same method, one can show
\begin{lemma} 
	\label{LemU12}	
	There holds $U^{1/2} =  \frac{1}{\omega} \tilde{H}^{1/2}(-1,1)$ with
	\[\forall u \in \tilde{H}^{\frac{1}{2}}(-1,1)\,, \quad \norm{u}_{\tilde{H}^{1/2}} \sim \norm{\frac{u}{\omega}}_{U^{1/2}}\,,\]
	and $U^{-1/2} = H^{-1/2}(-1,1)$ with 
	\[\forall u \in H^{-\frac{1}{2}}(-1,1)\,, \quad \norm{u}_{H^{-1/2}} \sim \norm{u}_{U^{-1/2}}\,.\]
\end{lemma}
\begin{corollary}
	\label{SkwNkwBij}
	The operators 
	\[S_{k,\omega_\Gamma} : T^{{-1}/{2}}(\Gamma) \to T^{1/2}(\Gamma)\, \quad \textup{and} \quad  N_{k,\omega_\Gamma} : U^{{1}/{2}}(\Gamma) \to U^{-1/2}(\Gamma)\,\]
	are bijective.
\end{corollary}
\subsection{Commutation relations}

To conclude this section, we recall the following commutations, proved in \cite{alougesAverseng}, which will be useful in the following. 
\begin{proposition}[{See \cite[Thm. 3]{alougesAverseng}}]
	For all $k \geq 0$, there holds
	\[S_{k,\omega} \left[-(\omega \partial_x)^2 -k^2 \omega^2\right] =\left[-(\omega \partial_x)^2 -k^2 \omega^2\right]S_{k,\omega}\,,\]
	\[N_{k,\omega} \left[-(\omega \partial_x)^2 -k^2 \omega^2\right] =\left[-(\omega \partial_x)^2 -k^2 \omega^2\right]N_{k,\omega}\,.\]
\end{proposition}
It is classical that the square root of an operator $A$ commutes with any operator that commute with $A$. Thus:
\begin{corollary}
	For all $k \geq 0$, there holds
	\label{commutSkNk}
	\[S_{k,\omega} \sqrt{-(\omega \partial_x)^2 -k^2 \omega^2} =\sqrt{-(\omega \partial_x)^2 -k^2 \omega^2}S_{k,\omega}\,,\]
	\[N_{k,\omega} \sqrt{-(\omega \partial_x)^2 -k^2 \omega^2} = \sqrt{-(\omega \partial_x)^2 -k^2 \omega^2}N_{k,\omega}\,.\]
\end{corollary}

\section{Parametrices for the weighted layer potentials}

We now apply the pseudo-differential theory on $T^s(\Gamma)$ and $U^s(\Gamma)$ to build low order parametrices for the weighted layer potentials $S_{k,\omega_\Gamma}$ and $N_{k,\omega_\Gamma}$ defined in the previous section. Asymptotic expansions are performed with the help of the symbolic calculus software Maple. The proofs of the next results are accompanied by commented Maple worksheets in Appendix \ref{AnnexeMaple}. 

\subsection{Helmholtz weighted single-layer}

\begin{lemma}
	\label{LemsymbolSk}
	The operator $S_{k,\omega_\Gamma}$ is in $\textit{Op}(S^{-1}_T(\Gamma))$. It satisfies	
	\[\mathcal{C}R_\Gamma S_{k,\omega_\Gamma}R_\Gamma^{-1} = \tilde{S}_k \mathcal{C} \quad \textup{in } T^{\infty}\]
	where $\tilde{S}_k$ is a classical elliptic PPDO, with a symbol given by
	\begin{equation}
	\begin{split}
	\sigma_{\tilde{S}_k}(\theta,\xi)&=\displaystyle \frac{1}{2\abs{\xi}}+\frac{k^2\abs{\Gamma}^2\sin(\theta)^2}{16\abs{\xi}^3}+{\frac {3i{k}^{2}{\abs{\Gamma}}^{2}\sin \theta \cos \theta }{16{\xi}^{4} \textup{sign}(\xi)}}\\
	+ &k^2\abs{\Gamma}^2\frac{-768 \kappa(\theta)^2 \abs{\Gamma}^2\sin^4\theta + 112 \sin^2 \theta + 3k^2 \abs{\Gamma}^2 \sin^4\theta - 48}{128 |\xi|^5}\\
	+ &\Sigma^{-6}\,.
	\end{split}
	\label{symboleSk}
	\end{equation}
	Recall that $R_\Gamma$ is the pullback associated to the parametrization $r$ of $\Gamma$, defined in section \ref{TsUs(Gamma)}. Moreover, here, $\kappa(\theta)$ denotes the curvature of $\Gamma$ at the point $r(\cos\theta)$. 
\end{lemma}
\begin{proof}
	The Hankel function admits the following expansion
	\begin{equation}
	\label{decompHankel}
	\frac{i}{4}H_0(z) = \frac{-1}{2\pi}\ln|z| J_0(z) + F_1(z^2)
	\end{equation}
	where $J_0$ is the Bessel function of first kind and order $0$ and where $F_1$ is analytic. 
	We fix a smooth function $u \in T^{\infty}$. Let $A = R_\Gamma S_{k,\omega_{\Gamma}}R_\Gamma^{-1}$. One has
	\[Au(x) = \int_{-1}^1 \frac{i}{4}H_0\left(k\abs{r(x) - r(y)}\right)\frac{u(y)}{\omega(y)}dy\,.\]
	Using the change of variables $x = \cos\theta$, $y = \cos \theta'$, we get 
	\[ Au(\cos\theta) = \int_{0}^\pi \frac{i}{4}H_0(k \abs{r(\cos \theta) - r(\cos \theta')}) u(\cos(\theta)) d\theta\,,\]
	which, in view of \eqref{decompHankel}, can be rewritten as
	\[\begin{split}
	Au(\cos \theta) = &\frac{-1}{2\pi}\int_{0}^\pi \ln \abs{\cos \theta - \cos \theta'} J_0(k \abs{r(\cos \theta) - r(\cos \theta')})\mathcal{C}u(\theta) d\theta \\
	&+ \int_{0}^\pi F_2(\cos\theta,\cos\theta') \mathcal{C}u(\theta) d\theta'
	\end{split}\]
	where 
	\[F_2(x,y) = -\frac{1}{2\pi}\ln \frac{\abs{r(x)-r(y)}}{\abs{x-y}} + F_1(k^2(x - y)^2)\,\]
	is a $C^\infty$ function. 
	By parity, the second integral defines an operator 
	\[Ku(\theta) = \frac{1}{2} \int_{-\pi}^{\pi} F_2(\cos\theta,\cos\theta') \mathcal{C}u(\theta) d\theta\,.\]
	There holds $K = \tilde{R}_1\mathcal{C}$ where, by \autoref{thrunen}, $R_1 \in \textit{Op}(\Sigma^{-\infty})$. For the first integral, we make the following classical manipulations. We first write $\cos \theta - \cos\theta' = - 2 \sin \frac{\theta + \theta'}{2}\sin \frac{\theta - \theta'}{2}$. Thus 
	$$\ln\abs{\cos \theta - \cos\theta'}= \ln \abs{\sqrt{2} \sin \frac{\theta + \theta'}{2}} + \ln \abs{\sqrt{2}\sin \frac{\theta - \theta'}{2}}\,.$$
	We then integrate and apply the change of variables $\theta \to - \theta$ for the second term, yielding
	\[Au(\cos\theta) = \left(\tilde{S}_{k,1} + \tilde{R}_1 \right)\mathcal{C}u(\theta)\]
	where
	\[\begin{split}
	\tilde{S}_{k,1}u(\theta) = &\frac{-1}{2\pi} \int_{-\pi}^\pi \ln \abs{\sqrt{2} \sin \frac{\theta - \theta'}{2}}J_0(k\abs{r(\cos\theta) - r(\cos \theta')})u(\theta')d\theta'\,.
	\end{split}\]
	Let $g : \theta \mapsto -\frac{1}{2\pi}\ln \abs{\sqrt{2} \sin \frac{\theta}{2}}$. It is well-known that $\hat{g}(n) = \frac{1}{2|n|}$ for $n \neq 0$. We may extend this by $\hat{g}(\xi) = \frac{1}{2|\xi|}$ away from $\xi = 0$. Let 
	$$a(\theta,\theta') = J_0\left(k\abs{r(\cos \theta ) - r(\cos \theta')}\right),$$ 
	which is a smooth function. By \autoref{thrunen}, the operator
	\[\tilde{S}_{k,1} u(\theta) \isdef \int_{-\pi}^\pi g(\theta-\theta') a(\theta,\theta')u(\theta')d\theta'\]
	is in $\textit{Op}(\Sigma^{-1})$ and is classical. Moreover, it is elliptic since $\hat{g}(\xi)$ does not vanish. In particular, $\tilde{S}_{k,1}u$ is a smooth function, from which we deduce that $\theta \mapsto Au(\cos\theta)$ is a smooth (even) function. For a smooth function, we have the expression
	\[Au(\cos\theta) = \mathcal{C} Au(\theta)\,.\] 
	This establishes that $\mathcal{C} Au = \tilde{S}_k \mathcal{C}u$ for any smooth function $u$. By \autoref{CharacSalpha}, this implies that $A \in \textit{Op}\left(S_T^{-1}\right)$, and equivalently, 
	$S_{k,\omega_\Gamma} \in Op(S^{-1}_T(\Gamma))$. We can compute the symbol of $\tilde{S}_{k,1}$ using the asymptotic expansion \eqref{FormuleIntegralOperatorSymbol}. The terms $\partial_s^ja(t,s)_{|t=s}$, can be related to the geometric characteristics of $\Gamma$ through eq.~\eqref{expansion_r}. The expansion \eqref{symboleSk} for $\tilde{S}_{k,1}$ is obtained with the help of Maple, and we refer the reader to Appendix \ref{AnnexeSk}, eq. (1). To simplify the expressions, the computations were only performed for $\xi > 0$, where $\hat{g}(\xi) = \frac{1}{2\xi}$. In the general case, the $\frac{1}{\xi^4}$ term in the asymptotic expansion must be multiplied by $\textup{sign}(\xi)$ to account for the case $\xi <0$. Obviously since $\tilde{R}_1 \in Op(\Sigma^{-\infty})$, the asymptotic expansion also holds for $\tilde{S}_k \isdef \tilde{S}_{k,1} + \tilde{R}_1$, concluding the proof. 
\end{proof}
\begin{corollary}
	The operator $S_{k,\omega_\Gamma}$ is elliptic and induces an isomorphism from $T^s(\Gamma)$ to $T^{s+1}(\Gamma)$ for all $s\in \R$. It thus maps $\Cinf(\Gamma)$ bijectively to itself. A pair of symbols of $S_{k,\omega_\Gamma}$ is given by
	\begin{align*}
	a_1(x,n) &= \frac{1}{2n} + \frac{k^2\abs{\Gamma}^2\omega(x)^2}{16n^3} + \tilde{a}_1\,,\\
	a_2(x,n) &= \frac{3xk^2 \abs{\Gamma}^2}{16 n^4} + \tilde{a}_2\,,
	\end{align*}
	where $(\tilde{a}_1,\tilde{a}_2) \in S^{-5}_T$.
\end{corollary}
\begin{proof}
	From the previous result and \autoref{PDOTs} {\em(i)}, we deduce that $S_{k,\omega_\Gamma}$ is continuous from $T^{s}(\Gamma)$ to $T^{s+1}(\Gamma)$. Furthermore, it can be written as
	\[S_{k,\omega_\Gamma} = R_\Gamma^{-1}S_{0,\omega}R_\Gamma + K\]
	where $K \in Op(S_T^{-3})$. Thus, $S_{k,\omega_\Gamma}$ is a compact perturbation of an isomorphism, thus a Fredholm operator of index $0$. Hence, it suffices show that if $u \in T^{s}(\Gamma)$ is such that $S_{k,\omega_\Gamma}u = 0$, then $u = 0$. For this purpose, notice that, since $\tilde{S}_k$ is elliptic, it follows from \autoref{PDOTs} {\em(v)} that $S_{k,\omega_\Gamma}$ is elliptic, and hence by \autoref{parametrixPDOTs} that it admits a parametrix $P \in Op(S_T^{1}(\Gamma))$. Applying $P$ on both sides of the equation $S_{k,\omega} u = 0$, it follows
	\[u = \left(I_d - PS_{k,\omega_\Gamma}\right)u\] 
	thus $u \in T^{\infty}(\Gamma)$. Second, we know that $S_{k,\omega_\Gamma}$ is a bijection from $T^{-1/2}(\Gamma)$ to $T^{1/2}(\Gamma)$ (cf. \autoref{SkwNkwBij}). Since $T^{\infty}(\Gamma) \subset T^{-{1}/{2}}(\Gamma)$, it follows that $u = 0$. Finally, the pair of symbols of $S_{k,\omega_\Gamma}$ is obtained from the symbol of $\tilde{S}_k$ using \autoref{PDOTs} {\em(iii)} 
\end{proof}
\begin{theorem}
	\label{TheSkomega}
	The operators $\left[-(\omega_\Gamma \partial_\tau)^2 - k^2\omega_\Gamma^2\right]$ and  $S_{k,\omega_\Gamma}$ satisfy
	\[\left[-(\omega_\Gamma \partial_\tau)^2 - k^2\omega_\Gamma^2\right] S_{k,\omega_\Gamma}^2 = \frac{I_d}{4} + T_{-4}\,.\]
\end{theorem}
\begin{proof}
	Using the symbolic calculus described in \autoref{sec:PDO}, one can compute an asymptotic expansion of the symbol of the pseudo-differential operator 
	\[\left[-(\omega_\Gamma \partial_\tau)^2 - k^2\omega_\Gamma^2\right]S_{k,\omega_\Gamma}^2 - \frac{I_d}{4}\,.\] 
	The symbol of this operator is found to be in $S^{-4}_T(\Gamma)$, from which the result follows. The first computation is detailed in the Appendix \ref{AnnexeSk}. 
\end{proof}
\begin{corollary}
	\label{optiTs}
	Let $A$ be a pseudo-differential operator of order $0$ and let
	\[\Delta_A = [-(\omega \partial_\tau)^2 + A]S_{k,\omega_\Gamma}^2 - \frac{I_d}{4}\,.\]
	Then $\Delta_A \in Op(S^{-2}_T(\Gamma))$ and 
	\[\Delta_A \in Op(S^{-4}_T(\Gamma)) \iff A = -k^2 \omega_\Gamma^2 + T_{-2}\,.\] 
\end{corollary}
\begin{proof}
	It suffices to write
	\[\Delta_A = \Delta_{-k^2\omega_{\Gamma}^2} + (A + k^2\omega^2)S_{k,\omega_\Gamma}^2\,.\]
	By the previous theorem, the first term in the rhs is in $Op(S^{-4}_T(\Gamma))$. Therefore, 
	\[\Delta_A \in Op(S_T^{-4}(\Gamma)) \iff (A + k^2\omega_\Gamma^2)S_{k,\omega_\Gamma}^2 \in Op(S^{-4}_T(\Gamma))\,.\]
	Using again a parametrix of ${S}_{k,\omega_\Gamma}$, we deduce
	\[\Delta_A \in Op(S_T^{-4}(\Gamma)) \iff (A + k^2\omega_\Gamma^2) \in Op(S^{-2}_T(\Gamma))\,\]
	which proves the result. 
\end{proof}
Recall the definition $[A,B] \isdef AB - BA$. 
\begin{lemma}
	\label{lemCommutator}
	The commutator
	\[\left[ S_{k,\omega_\Gamma}, \sqrt{-(\omega\partial_\tau)^2 -k^2 \omega_{\Gamma}^2}\right]\]
	is in $Op(S^{-5}_T(\Gamma))$. 
\end{lemma}
\begin{proof}
	By considering the symbols of the operators, one can check that
	\[R_\Gamma S_{k,\omega_{\Gamma}}R_\Gamma^{-1} = S_{|\Gamma|k,\omega} + K\]
	where $K \in Op(S_T^{-5})$. Moreover, 
	\[R_\Gamma(-(\omega\partial_\tau)^2 -k^2 \omega_{\Gamma}^2)R_\Gamma^{-1} = -(\omega \partial_x)^2 - (k|\Gamma|)^2 \omega^2\,.\]
	Recalling the commutation relation stated in \autoref{commutSkNk}, this gives
	\[R_\Gamma\left[ S_{k,\omega_\Gamma}, -(\omega\partial_\tau)^2 -k^2 \omega_{\Gamma}^2\right]R_\Gamma^{-1} = \left[ K, -(\omega\partial_x)^2 -(k \abs{\Gamma})^2 \omega^2\right]\,.\]
	Obviously, $-(\omega \partial_x)^2 - (k\abs{\Gamma})^2 \omega^2 \in Op(S^{2}_T)$ so the result follows from \autoref{PDOTs} {\em (iv)}.
\end{proof}
\begin{theorem}
	\label{resultatChaud}
	There holds
	\[\sqrt{-(\omega_\Gamma \partial_\tau)^2 - k^2 \omega_\Gamma^2}S_{k,\omega_\Gamma} = \frac{I_d}{2} + T_{-4}\,.\]
\end{theorem}
\begin{proof}
	One has
	\[S_{k,\omega_\Gamma}^2 \left[-(\omega_\Gamma \partial_\tau)^2 - k^2 \omega_\Gamma^2\right] = \frac{I_d}{4} + T_{-4}\,.\]
	Let us denote 
	\[A = S_{k,\omega_\Gamma} \sqrt{-(\omega_\Gamma \partial_\tau)^2 - k^2 \omega_\Gamma^2}\,.\] 
	Exploiting \autoref{lemCommutator}, we deduce that
	\[K \isdef \left(A + \frac{I_d}{2}\right)\left(A - \frac{I_d}{2}\right)\]
	is in $Op(S^{-4}_{T}(\Gamma))$. 
	Furthermore, one has
	\[\mathcal{C}R_\Gamma\left(A + \frac{I_d}{2}\right)R_\Gamma^{-1} = \left(\tilde{S_k} \sqrt{\tilde{D}} + \frac{I_d}{2} \right)\mathcal{C} \quad \textup{in } T^{\infty}\]
	by \autoref{lemCalculDesSqr}. Notice that the PPDO 
	$\tilde{B} = \tilde{S}_k \sqrt{\tilde{D}} + \frac{I_d}{2}$
	is classical of order $0$ with a principal symbol given by $\sigma_{0}(\theta,n) = 1$. Thus, it is elliptic. Therefore, by \autoref{parametrixPDOTs}, $A + I_d$ admits a parametrix $P$ of order $0$. If $R$ is the smoothing operator such that
	\[P(A + I_d) = I_d + R\,,\]
	we then have
	\[A-I_d = PK - R(A + I_d)\,.\]
	It is straightforward to check that the operator on the rhs is of order $-4$, from which the result follows.   
\end{proof}
\subsection{Neumann problem}

We saw in \autoref{NkomegaSkomega} that $N_{k,\omega_\Gamma} = N_1 - k^2N_2$ where 
\[N_1 = -\partial_\tau S_{k,\omega} \omega_\Gamma \partial_\tau \omega_\Gamma\]
and $N_2 = V_k\omega_\Gamma^2$ with
\[V_ku(x) = \int_{\Gamma} \frac{G_k(x-y)N(x)\cdot N(y) u(y)}{\omega_\Gamma(y)}d\sigma_y\,.\]
\begin{lemma}
	\label{LemsymbolN1}
	The operator $N_1$ is in $\textit{Op}(S_U^2(\Gamma))$ and  
	\[\mathcal{S}R_\Gamma N_1R_\Gamma^{-1} = \tilde{N}_1 \mathcal{S} \quad \textup{ in } U^{\infty}\]
	where $\tilde{N}_1$ is a classical and elliptic PPDO with a symbol $\sigma_{\tilde{N}_1}$ satisfying
	\begin{equation}
	\sigma_{\tilde{N}_1}(\theta,\xi) = \frac{|\xi|}{2} + \frac{1}{16}\frac{k^2 \abs{\Gamma}^2 \sin^2(\theta)}{|\xi|} + i\frac{k^2 \abs{\Gamma}^2 \sin\theta\cos\theta}{16\xi^2\textup{sign}(\xi)} + \Sigma^{-3}\,.
	\label{symboleN1}
	\end{equation}
	As a consequence, $N_1$ admits the pair of symbols
	\begin{align}
	a_1(x,n) &= \frac{n}{2} + \frac{k^2 \abs{\Gamma}^2 \omega(x)^2}{16n} + \tilde{a}_1\,,\\
	a_2(x,n) &= \frac{k^2\abs{\Gamma}^2x}{16n^2} + \tilde{a}_2\,,
	\end{align}
	where $(\tilde{a}_1,\tilde{a}_2) \in S^{-3}_U$.
\end{lemma}
\begin{proof}
	This result is obtained by symbolic calculus combining \autoref{LemsymbolSk} and \autoref{LemdxAomegadeomega}. See appendix \ref{AnnexeNk} eq. (1). 
\end{proof}
\noindent A small adaptation of the proof of \autoref{LemsymbolSk} yields the following result (see Appendix \ref{AnnexeNk} eq. (2))
\begin{lemma}
	\label{LemsymbolVk}
	The operator $V_k$ is in $\textit{Op}(S_T^{-1}(\Gamma))$ and 
	\[\mathcal{C}R_\Gamma V_kR_\Gamma^{-1} = \tilde{V}_k \mathcal{C} \quad \textup{in } T^{\infty}\]
	where $\tilde{V}_k$ is a classical and elliptic PPDO with a symbol $\sigma_{\tilde{V}_k}$ satisfying
	\[\sigma_{\tilde{V}_k} = \frac{1}{2|\xi|} + \Sigma^{-3}\,.\]
	$V_k$ thus admits the pair of symbols
	\begin{align*}
	a_1(x,n) &= \frac{1}{2n} + \tilde{a}_1(x,n)\,,\\
	a_2(x,n) &= \tilde{a}_2(x,n)\,,
	\end{align*}
	where $(\tilde{a}_1,\tilde{a}_2) \in S^{-3}_T$.	
\end{lemma}
\noindent Applying \autoref{LemAomega2}, we deduce
\begin{corollary}
	The operator $N_2$ is in $\textit{Op}(S^{-1}_U(\Gamma))$ and satisfies
	\[\mathcal{S} R_\Gamma N_2R_\Gamma^{-1} = \tilde{N}_2 \mathcal{S} \quad \textup{in } U^{\infty}\]
	where $\tilde{N}_2$ is a classical PPDO with a symbol $\sigma_{\tilde{N}_2}$ satisfying
	\begin{equation}
	\sigma_{\tilde{N}_2} = \frac{\sin^2\theta}{2\abs{\xi}} + i\frac{\sin\theta \cos \theta}{2 \xi^2\textup{sign}(\xi)} + \Sigma^{-3}\,.
	\label{symboleN2}
	\end{equation}
	A pair of symbols for $N_2$ is thus
	\begin{align*}
	a_1(x,n) &= \frac{\omega(x)^2}{2n} + \tilde{a}_1 \,,\\
	a_2(x,n) &= \frac{x}{2n^2} + \tilde{a}_2\,,
	\end{align*}
	where $(\tilde{a}_1,\tilde{a}_2) \in S^{-3}_U$. 
\end{corollary}
\begin{theorem}
	\label{TheNkomega}
	The operato  $N_{k,\omega_\Gamma}$ is elliptic and satisfies
	\[N_{k,\omega_\Gamma}^2 = \frac{1}{4}\left[-(\partial_\tau \omega_\Gamma)^2 - k^2 \omega_\Gamma^2\right] + U_{-2}\,. \]
\end{theorem}
\begin{proof}
	We have
	\[\mathcal{C}R_\Gamma N_{k,\omega} R_\Gamma^{-1} = \left(\tilde{N_1} - k^2 \tilde{N_2}\right)\mathcal{C} \quad \textup{ in } U^{\infty}\]
	where both $\tilde{N_1}$ and $\tilde{N_2}$ are classical. Moreover, the principal symbol of $\tilde{N_1} - k^2 \tilde{N_2}$ is $\frac{\xi}{2}$, thus $N_{k,\omega_\Gamma}$ is elliptic by \autoref{PDOUs} {\em(v)}. 
	We have asymptotic expansions available for the symbols of the operators $N_{k,\omega_\Gamma}$ and $\left[-(\partial_\tau \omega_\Gamma )^2 - k^2\omega_\Gamma^2\right]$. We can thus compute an asymptotic expansion of the symbol of the operator $N_{k,\omega_\Gamma}^2 - \left[-(\partial_\tau \omega_\Gamma )^2 - k^2\omega_\Gamma^2\right]$
	which turns out to be in $S^{-2}_U(\Gamma)$, giving the result. The details of the computations can be found in Appendix \ref{AnnexeNk}. 
\end{proof}
\noindent Reasoning with a parametrix as in the previous section one can prove the following results
\begin{corollary}
	\label{optiUs}
	Let $A \in Op(S_U^{0}(\Gamma))$ and let 
	\[\Delta_A = N_{k,\omega_\Gamma}^2 - \frac{1}{4}\left[-(\partial_\tau \omega_\Gamma )^2 + A\right]\,.\]
	Then $\Delta_A \in Op(S_U^{0}(\Gamma))$ and moreover
	\[\Delta_A \in Op(S^{-2}_U(\Gamma)) \iff A = -k^2 \omega_\Gamma^2 + U_{-2}\,.\]
\end{corollary}
\begin{theorem}
	There holds
	\[N_{k,\omega_\Gamma} = \frac{1}{2}\sqrt{-(\partial_\tau \omega_\Gamma)^2 - k^2 \omega_\Gamma^2} + U_{-3}\,.\]
\end{theorem}

\section{Conclusion}

In this work, we have set forth a pseudo-differential analysis for the single and hypersingular (weighted) layer potentials on open curves, $S_{k,\omega_\Gamma}$ and $N_{k,\omega_\Gamma}$. By this analysis, we can recover the symbols of the layer potentials and perform some manipulations that are usual in the domain of pseudo-differential calculus. This allows to prove that the operators 
\[P_1 = \sqrt{-(\omega_\Gamma \partial_\tau)^2 -k^2 \omega^2_\Gamma}, \quad P_2 = [-(\partial_\tau\omega_\Gamma )^2 -k^2 \omega^2_\Gamma]^{-\frac{1}{2}}\]
are low order parametrices for the layer potentials $S_{k,\omega_\Gamma}$ and $N_{k,\omega_\Gamma}$ respectively, and justifies the preconditioning method exposed in \cite{alougesAverseng}. In particular, \autoref{optiTs} and \autoref{optiUs} show that the terms $-k^2 \omega_\Gamma^2$ in $P_1$ and $P_2$ are the best one among other possible first order corrections. 

The analysis heavily relies on some explicit formulas available for $k = 0$, namely 
\begin{equation}
\label{explLaplace}
S_{0,\omega} T_n = \frac{1}{2n} T_n, \quad N_{0,\omega} U_n = \frac{(n+1)}{2}U_n\,.
\end{equation}
Other works have exploited those relations to build closed form inverses for the Laplace potentials $S_{0}$ and $N_{0}$ e.g. \cite{jerez2012explicit}. For $k \neq 0$, the authors of the previous work have suggested to use the Laplace inverses as preconditioners for the Helmholtz layer potentials. We believe that, to refine this approach and correctly capture the suitable dependence in $k$ of the operators, the pseudo-differential route is almost unavoidable, since explicit formulas comparable to \eqref{explLaplace} are not known as soon as $k \neq 0$. We have demonstrated in \cite{alougesAverseng} the importance of such a correction in $k$ in several numerical examples. 

\noindent Possible future directions include
\begin{itemize}
	\item[-] The generalization of the approach presented here to 3 space dimensions. The pseudo-differential calculus on the 2-sphere is not as simple as the one on the torus, but the latter has mainly been used here, in place of the fully general pseudo-differential theory on manifolds, as a convenience to simplify the presentation and especially avoid coordinate charts. Thus, this program should be realizable without too many difficulties. 
	\item[-] The analysis of the preconditioning strategy proposed in \cite{bruno2012second}. It could be possible, by symbolic calculus, to compute the symbol of $S_{k,\omega_\Gamma} N_{k,\omega_\Gamma}$ although the two operators do not belong to the same scales. This is the object of ongoing work. We expect that the remainder is at least of order $-4$, as suggested by the numerical results exposed in \cite{alougesAverseng}. 
\end{itemize}

\paragraph{Acknowledgement:}
	I wish to thank Pr. Fran\c{c}ois Alouges for his patient support and valuable help during the elaboration of this work. I also wish to thank Pr. Ralf Hiptmair for his helpful advices regarding the presentation.

\section{ Appendix : Symbolic Calculus}

\label{AnnexeMaple}

\subsection{Single layer potential}

\label{AnnexeSk}

\noindent\includegraphics[scale = 0.6]{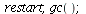}

\noindent Procedure for the (usual) symbolic calculs:\\

\noindent
\includegraphics[scale = 0.6]{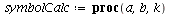}\\\includegraphics[scale = 0.6]{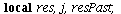}\\\includegraphics[scale = 0.6]{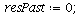}\\\includegraphics[scale = 0.6]{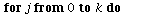}\\\includegraphics[scale = 0.6]{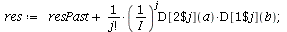}\\\includegraphics[scale = 0.6]{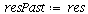}\\\includegraphics[scale = 0.6]{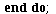}\\\includegraphics[scale = 0.6]{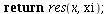}\\\includegraphics[scale = 0.6]{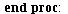}~\\

\noindent
Symbol of the operator $-(\omega \partial_x)^2 - k^2\omega^2)$

\noindent
\includegraphics[scale = 0.6]{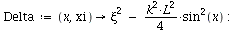}~

\noindent Taylor expansion of $|r(\cos t) - r(\cos s)|$ where $r : [-1,1] \to \Gamma$ is a parametrization such that $|r'(x)| = \frac{L}{2}$ for all $x$. $L$ is the length of $\Gamma$ and $C(t)$ is the curvature. $K$ denotes the unknown constant in the next order. 

\noindent
\includegraphics[scale = 0.6]{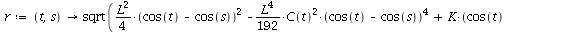}\\\includegraphics[scale = 0.6]{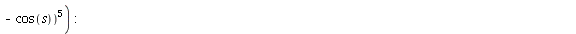}~

\noindent Taylor expansion of $b(t,s) =
J_0(k|r(\cos t) - r(\cos s)|)$. 

\includegraphics[scale = 0.6]{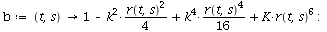}~

\noindent We use the following procedure to compute an asymptotic
expansion of the symbol of $\tilde{S}_k$.
$K$ denotes the unknown constant in the
next order. \\

\noindent \textbf{\includegraphics[scale = 0.6]{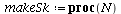}\\\includegraphics[scale = 0.6]{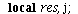}\\\includegraphics[scale = 0.6]{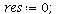}\\\includegraphics[scale = 0.6]{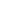}\\\includegraphics[scale = 0.6]{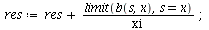}\\\includegraphics[scale = 0.6]{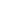}\\\includegraphics[scale = 0.6]{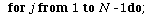}\\\includegraphics[scale = 0.6]{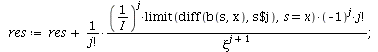}\\\includegraphics[scale = 0.6]{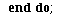}\\\includegraphics[scale = 0.6]{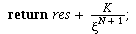}\\\includegraphics[scale = 0.6]{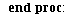}}~

\noindent We obtain the following asymptotic expansion up to order
$6$:

\noindent \includegraphics[scale = 0.6]{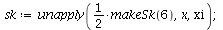}~

\begin{longtable}[c]{@{}ll@{}}
	\toprule\addlinespace
	\includegraphics[scale = 0.6]{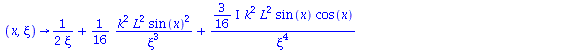}\\\includegraphics[scale = 0.6]{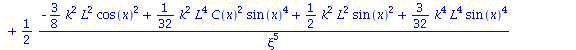}\\\includegraphics[scale = 0.6]{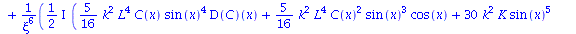}\\\includegraphics[scale = 0.6]{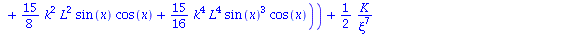}
	& (1)
	\\\addlinespace
	\bottomrule
\end{longtable}

\noindent We can thus compute the symbol of the operator
$\tilde{S}_k^2$
using symbolic calculus, and keep the terms up to order
$-6$.~

\noindent \textbf{\includegraphics[scale = 0.6]{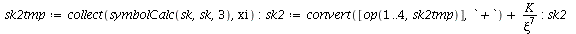}\\\includegraphics[scale = 0.6]{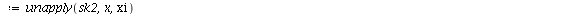}}~

\begin{longtable}[c]{@{}ll@{}}
	\toprule\addlinespace
	\includegraphics[scale = 0.6]{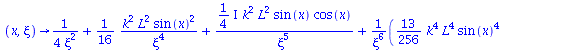}\\
	\includegraphics[scale = 0.6]{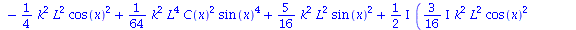}\\
	\includegraphics[scale = 0.6]{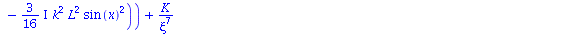}
	& (2)
	\\\addlinespace
	\bottomrule
\end{longtable}

\noindent We now apply symbolic calculus to compute an asymptotic expansion of the
symbol of the composition 
$\tilde{D}_k\tilde{S}_k^2$, keeping only the
first two terms~

\noindent \textbf{\includegraphics[scale = 0.6]{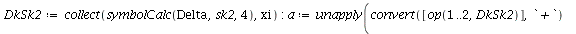}\\\includegraphics[scale = 0.6]{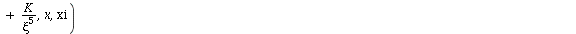}}~

\begin{longtable}[c]{@{}ll@{}}
	\toprule\addlinespace
	\includegraphics[scale = 0.6]{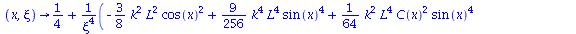}\\\includegraphics[scale = 0.6]{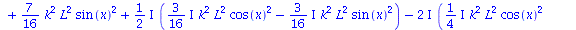}\\\includegraphics[scale = 0.6]{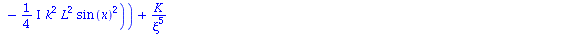}
	& (3)
	\\\addlinespace
	\bottomrule
\end{longtable}

\noindent We see that this is of the form
$\frac{1}{4} + \sigma$ where $\sigma \in \Sigma^{-4}$.

\subsection{Hypersingular operator}

\label{AnnexeNk}

\noindent \textbf{\includegraphics[scale = 0.6]{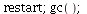}}\\

\noindent Procedure for the (usual) symbolic calculs:\\

\noindent \textbf{\includegraphics[scale = 0.6]{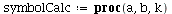}\\\includegraphics[scale = 0.6]{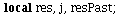}\\\includegraphics[scale = 0.6]{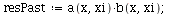}\\\includegraphics[scale = 0.6]{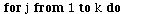}\\\includegraphics[scale = 0.6]{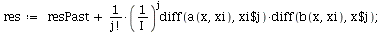}\\\includegraphics[scale = 0.6]{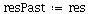}\\\includegraphics[scale = 0.6]{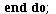}\\\includegraphics[scale = 0.6]{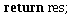}\\\includegraphics[scale = 0.6]{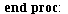}}\\

\noindent Symbols of the operators \( (\partial_x
\omega)^2 - k^2 \omega^2
\), \(u(x) \mapsto
\sin(x) u(x)\) and \(
\partial_\theta\):\\

\noindent \textbf{\includegraphics[scale = 0.6]{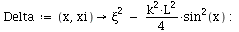}}

\noindent \textbf{\includegraphics[scale = 0.6]{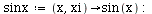}}

\noindent \textbf{\includegraphics[scale = 0.6]{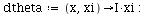}}\\

\noindent Taylor expansion of \(|r(\cos t)
- r(\cos s)|\) where
\(r :[-1,1] \to
\Gamma\) is a parametrization such that \( |r'(x)| = \frac{L}{2}\) for all \(x\). \(L\) is the length of \( \Gamma \) and \(C(t)\) is the curvature. \(K\) denotes the unknown constant in the next order:\\

\noindent \textbf{\includegraphics[scale = 0.6]{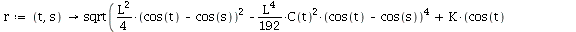}\\\includegraphics[scale = 0.6]{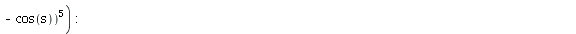}}\\

\noindent Taylor expansion of \( b(t,s) = J_0(k
|r(\cos t) - r(\cos s)|)
n(\cos t) \cdot n(\cos s)
\) where \(n(t)\) is the
normal vector at the point \(r(t)\):\\

\noindent \textbf{\includegraphics[scale = 0.6]{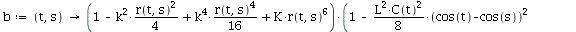}\\\includegraphics[scale = 0.6]{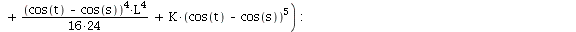}}\\

\noindent We first compute an asymptotic expansion up to order
\(3\) of the symbol of \(
\tilde{N}_1 =
-\partial_{\theta}\tilde{S}_k
\partial_{\theta}\). We
already know the symbol of \(
\tilde{S}_k \)\\

\noindent \textbf{\includegraphics[scale = 0.6]{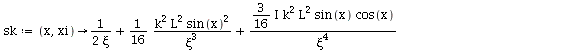}\\\includegraphics[scale = 0.6]{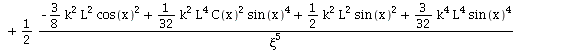}\\\includegraphics[scale = 0.6]{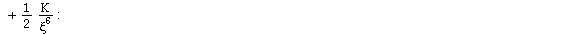}}\\

\noindent so we just need to use the usual symbol calculus. We obtain the
following symbol:

\noindent \textbf{\includegraphics[scale = 0.6]{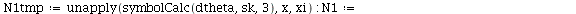}\\\includegraphics[scale = 0.6]{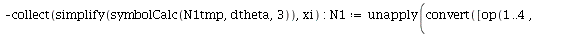}\\\includegraphics[scale = 0.6]{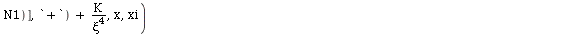}}

\begin{longtable}[c]{@{}ll@{}}
	\toprule\addlinespace
	\includegraphics[scale = 0.6]{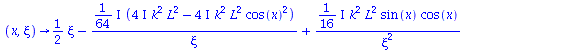}\\\includegraphics[scale = 0.6]{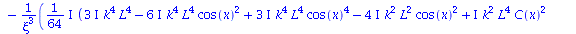}\\\includegraphics[scale = 0.6]{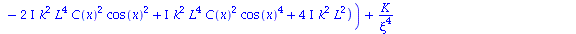}
	& (1)
	\\\addlinespace
	\bottomrule
\end{longtable}

\noindent We then turn to the computation of an asymptotic expansion of the symbol of \(\tilde{N}_2\). We start with a procedure to compute the symbol of the operator
\(\tilde{V}_k\):\\

\noindent \textbf{\includegraphics[scale = 0.6]{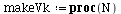}\\\includegraphics[scale = 0.6]{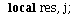}\\\includegraphics[scale = 0.6]{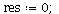}\\\includegraphics[scale = 0.6]{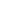}\\\includegraphics[scale = 0.6]{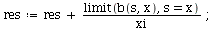}\\\includegraphics[scale = 0.6]{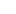}\\\includegraphics[scale = 0.6]{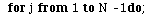}\\\includegraphics[scale = 0.6]{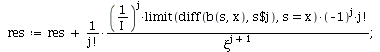}\\\includegraphics[scale = 0.6]{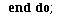}\\\includegraphics[scale = 0.6]{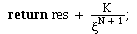}\\\includegraphics[scale = 0.6]{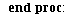}}\\

\noindent The following symbol is obtained for
\(\tilde{V}_k\):

\noindent \textbf{\includegraphics[scale = 0.6]{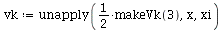}}

\begin{longtable}[c]{@{}ll@{}}
	\toprule\addlinespace
	\includegraphics[scale = 0.6]{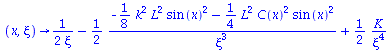} & (2)
	\\\addlinespace
	\bottomrule
\end{longtable}

\noindent The operator
\(\tilde{N}_2\) is then
obtained by multiplying left and right by the operator
\(u(x) \mapsto u(x)sin(x)
\). We obtain the following asymptotic expantion up to
order \(3\).\\

\noindent \textbf{\includegraphics[scale = 0.6]{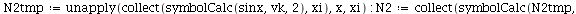}\\\includegraphics[scale = 0.6]{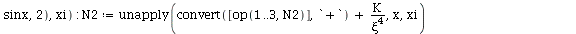}}

\begin{longtable}[c]{@{}ll@{}}
	\toprule\addlinespace
	\includegraphics[scale = 0.6]{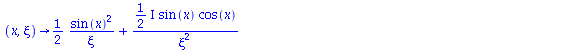}\\\includegraphics[scale = 0.6]{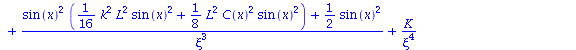}
	& (3)
	\\\addlinespace
	\bottomrule
\end{longtable}

\noindent The symbol of \( \tilde{N}_k =
\tilde{N}_1 - \frac{k^2
	L^2}{4} \tilde{N}_2\) is thus,
retaining only the terms up to order
\(3\):\\

\noindent \textbf{\includegraphics[scale = 0.6]{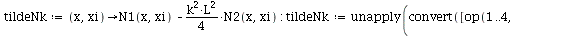}\\\includegraphics[scale = 0.6]{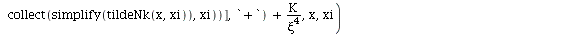}}

\begin{longtable}[c]{@{}ll@{}}
	\toprule\addlinespace
	\includegraphics[scale = 0.6]{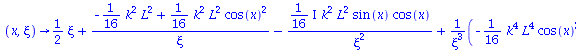}\\\includegraphics[scale = 0.6]{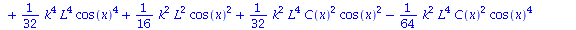}\\\includegraphics[scale = 0.6]{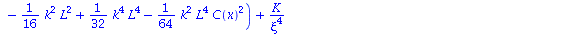}
	& (4)
	\\\addlinespace
	\bottomrule
\end{longtable}

\noindent We can now compute the symbol of
\(\tilde{N}_k^2 \) by
usual symbolic calculus, retaining terms up to order
\(2\).\\

\noindent \textbf{\includegraphics[scale = 0.6]{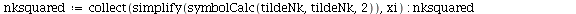}\\\includegraphics[scale = 0.6]{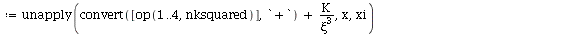}}

\begin{longtable}[c]{@{}ll@{}}
	\toprule\addlinespace
	\includegraphics[scale = 0.6]{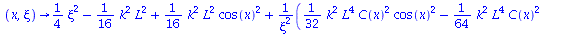}\\\includegraphics[scale = 0.6]{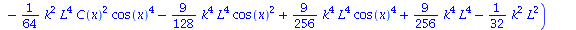}\\\includegraphics[scale = 0.6]{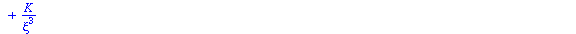}
	& (5)
	\\\addlinespace
	\bottomrule
\end{longtable}

\noindent The difference \( \tilde{N}_k^2 -
\frac{1}{4}\tilde{D}_k
\),\\

\noindent \textbf{\includegraphics[scale = 0.6]{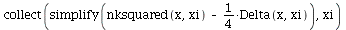}}

\begin{longtable}[c]{@{}ll@{}}
	\toprule\addlinespace
	\includegraphics[scale = 0.6]{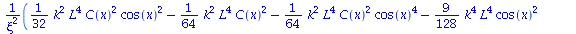} \\
	\includegraphics[scale = 0.6]{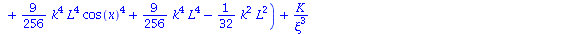}
	& (6)
	\\\addlinespace
	\bottomrule
\end{longtable}

\noindent is in \(\Sigma^{-2}\)

\end{document}